\documentclass[11pt,twoside,a4paper]{article}

\usepackage{amssymb}

\usepackage{artmods,macros}
\usepackage{amsmath}
\usepackage{subfigure}
\usepackage{booktabs}
\usepackage{rotating}
\usepackage{enumerate}

\usepackage[font=small,labelfont=bf,labelsep=period,margin=0.5cm]{caption}

\input macros

\usepackage{natbib}
 \bibpunct[, ]{(}{)}{,}{a}{}{,}%
 \def\newblock{\ }%
\providecommand{\abs}[1]{\left\lvert #1 \right\rvert}
\DeclareMathOperator{\closure}{cl}

\def\blackslug{\hbox{\hskip 1pt \vrule width 4pt height 6pt depth 1.5pt
  \hskip 1pt}}
\def\Re{\mathbb{R}}
\def\Ex{\mathbb{E}}
\def\allP{{\Pi}}
\def\A{\mathcal{A}}

\def\R{\mathcal{R}}
\def\U{\mathcal{U}}
\def\S{\mathcal{S}}

\def\T{\mathcal{T}}
\def\X{\mathcal{X}}
\def\Dirac{\mathcal{D}}
\def\dhat{\hat d}

\usepackage[breaklinks]{hyperref}
\hypersetup{ colorlinks,
linkcolor=blue,
filecolor=green,
urlcolor=blue,
citecolor=blue }

  \title{Necessary and sufficient conditions for Pareto efficiency in robust multiobjective optimization}

\author{Rasmus Bokrantz\thanks{E-mail: {\tt bokrantz@kth.se} and {\tt rasmus.bokrantz@raysearchlabs.com}}~$^\dagger$ and Albin Fredriksson\thanks{Optimization and Systems Theory, Department of Mathematics, KTH Royal Institute of Technology, SE-100 44 Stockholm, Sweden; and RaySearch Laboratories, Sveav\"{a}gen 25, SE-111 34 Stockholm, Sweden.}}

\date  {May, 2017}

\begin{document}

\maketitle

\vspace{-0.3cm}
\noindent
\textbf{This is a post-print of an article published in the European Journal of Operational Research. Please cite as:\\\\Bokrantz, R., A.~Fredriksson. 2017. Necessary and sufficient conditions for Pareto efficiency in robust multiobjective optimization. \emph{Eur.~J.~Oper.~Res.}, 262(2) 682--692. }

\vspace{0.3cm}

\begin{abstract}
\noindent
We provide necessary and sufficient conditions for robust efficiency (in the sense of Ehrgott et al. (2014)) to multiobjective optimization problems that depend on uncertain parameters. These conditions state that a solution is robust efficient (under minimization) if it is optimal to a strongly increasing scalarizing function, and only if it is optimal to a strictly increasing scalarizing function. By counterexample, we show that the necessary condition cannot be strengthened to convex scalarizing functions, even for convex problems. We therefore define and characterize a subset of the robust efficient solutions for which an analogous necessary condition holds with respect to convex scalarizing functions. This result parallels the deterministic case where optimality to a convex and strictly increasing scalarizing function constitutes a necessary condition for efficiency. By a numerical example from the field of radiation therapy treatment plan optimization, we illustrate that the curvature of the scalarizing function influences the conservatism of an optimized solution in the uncertain case.
\end{abstract}





\noindent

\section{Introduction}
Optimization problems that arise in applications often rely on parameter values that are unknown at the time when the problems are solved. A risk therefore exists that if the actual parameter values deviate from the estimated ones, an optimized solution can be suboptimal or even infeasible. It is also common that the performance of the solutions is judged by multiple objectives that are in conflict, such as quality versus cost. This situation calls for a prioritization of the objectives, which can be difficult to articulate into a precise mathematical function prior to the procurement of knowledge about the possible solutions to the problem.

In this paper, we address the twin difficulties of uncertainty in problem data and uncertainty in one's preferences by combining robust and multiobjective optimization. Robust optimization finds solutions that perform well also under perturbed parameter values. Different robustness concepts deem different solutions to be robust optimal. In this paper, we consider \emph{minimax robustness}, in which the optimization is conditioned on the worst case realization of the uncertainty within some uncertainty set, see, e.g.,~\citet{ben-tal09}. The goal in multiobjective optimization is generally to find the \emph{Pareto efficient} (or simply \emph{efficient}) solutions, meaning the feasible solutions such that no objective can be improved without a sacrifice in one of the others, see, e.g.,~\citet{miettinen99} and~\citet{ehrgott05}. A representation of the efficient solutions facilitates decision making in the sense that it enables the decision maker to explore the possible tradeoffs between the objectives before deciding how they should be prioritized.

Robust optimization and multiobjective optimization are both well-studied topics, but rarely have they been considered in combination. An obstacle that may have prevented their simultaneous application is that the notion of a worst case is not well-defined unless the priorities of the different objectives have been decided (which removes the multiobjective aspect of the problem). This difficulty renders the standard definition of efficiency unable to characterize the robust (efficient) solutions to uncertain multiobjective problems. Several previous investigators circumvent this difficulty by optimizing objectives that are averaged over a neighborhood of the current solution, or by optimizing the nominal objective values subject to constraints that limit the objectives' variability over the uncertainty set, see, e.g.,~\citet{gunawan05} and~\citet{deb06}. Another possibility that has been explored by~\citet{kuroiwa12},~\citet{chen12}, and~\citet{fliege14} is to let each objective evaluate with respect to its own worst case realization of the uncertainty. This objectivewise formulation fits within the standard multiobjective optimization framework, but is conservative in the sense that it protects against simultaneous occurrence of several realizations of the uncertainty. We demonstrate by a numerical example that such conservatism is unwarranted if the uncertainty in reality affects the objectives in a correlated manner, because it then sacrifices quality under errors that can occur in practice in favor of robustness against hypothetical errors where the uncertainty materializes independently for the different objectives.

In the present paper, we are concerned with uncertain multiobjective problems where the uncertainty jointly affects all objectives. Our paper extends the research of~\citet{ehrgott14}, who proposed a generalization of efficiency, called \emph{robust Pareto efficiency} (or simply \emph{robust efficiency}), which is applicable to uncertain multiobjective optimization problems on the form that we consider. Since the publication of a preprint of the present paper~\citep{bokrantz13b}, a number of related publications has appeared: different concepts for robust efficiency are reviewed in~\citet{ide16}; optimality and duality for uncertain multiobjective problems are considered in~\citet{chuong16}; robust counterparts are derived in~\citet{wang15}; numerically tractable optimality conditions for uncertain multiobjective linear problems are given in~\citet{goberna15}; and the relation to the field of set-valued optimization are explored in~\citet{ide14} and~\citet{ide14b}.

In their paper, \citet{ehrgott14} discuss how to find robust efficient solutions numerically, and give proof that two standard methods for calculation of efficient solutions to deterministic problems---the weighted sum method and the $\epsilon$-constraint method---can be used also to compute robust efficient solutions~\citep[Theorems 4.3 and 4.7]{ehrgott14}. None of the two methods is, however, capable of finding the full robust efficient set~\citep[Examples 4.5 and 4.9]{ehrgott14}. This result is in sharp contrast to the deterministic case where the weighted sum method finds all efficient solutions if the problem is convex and the $\epsilon$-constraint method finds all efficient solutions in general~\citep[Theorems 3.1.4 and 3.2.2]{miettinen99}. The purpose of the present paper is to gain better insight to how robust efficient solutions can be computed, both in general and as solutions to tractable (i.e., convex) optimization problems. Specifically, our main contributions are the following:
\begin{itemize}
\item We give necessary and sufficient conditions for robust efficiency to uncertain multiobjective programs. These conditions show that all robust efficient solutions can be found by minimization of strictly increasing scalarizing functions, and that all optimal solutions to strongly increasing scalarizing functions are robust efficient.

\item We introduce a more restrictive concept of efficiency than robust efficiency, called \emph{convex hull robust Pareto efficiency} (or simply \emph{convex hull efficiency}), and give necessary and sufficient conditions for convex hull efficiency to uncertain multiobjective programs. These conditions are analogous to those for robust efficiency, except that they hold with respect to convex scalarizing functions. The necessary condition thereby parallels a classical result for the deterministic case, which states that each efficient solution is optimal to some strictly increasing and convex scalarizing function. Further, if the optimization problem is convex for any fixed value of the uncertain parameter, then our results assert that each convex hull efficient solution is optimal to a convex scalarized problem, something that does not hold for robust efficient solutions in general.

\end{itemize}

We also illustrate, by application to optimization of proton therapy for cancer treatment, that the properties of the scalarizing function influence the conservatism of the attainable set of efficient solutions. This fact holds true also for scalarizing functions that have identical attainable sets of efficient solutions in the deterministic case.

\section{Notation}\label{sec-notation}
Vector inequalities are understood componentwise: it holds that $a \leq b$ if $a$ is less than or equal to $b$ in every component, and $a < b$ if $a$ is strictly less than $b$ in every component. The difference $A-B$ of two sets $A$ and $B$ denotes \hbox{$\{a-b:a \in A, b \in B\}$} while the difference \hbox{$a - B$} between $B$ and an element $a$ of $A$ denotes \hbox{$\{a\} - B = \{a - b: b \in B\}$}. The convex hull of a possibly infinite set $A$ is denoted $\conv(A)$ and defined as
\[
\conv(A) = \left\{ \sum_{i \in I} \lambda_i a_i : a_i \in A,~\lambda_i > 0,~\sum_{i \in I} \lambda_i = 1,~\abs{I} < \infty \right\}.
\]
The interior of a set $A$ is denoted $\interior(A)$ and the closure of a set $A$ is denoted $\closure(A)$. A semicolon is used to separate variables from parameters in the arguments to a function: $f(x;a)$ has the variable $x$ and the parameter $a$. A function is said to be convex if it is convex in its variables, but not necessarily in its parameters. The image $f(A)$ of a set $A$ in the domain of a function $f$ denotes \hbox{$\{f(a): a \in A\}$}. Similarly, $f(x;A)$ for a set of parameters $A$ of a function $f$ denotes \hbox{$\{f(x;a): a \in A \}$}. The following definitions are used to characterize monotonicity:
\begin{definition}[Increasing]
A function $u:\Re^n \to \Re$ is \emph{increasing} if for \hbox{$y,y' \in \Re^n$}
\[
y \leq y' \textrm{ implies } u(y) \leq u(y').
\]
\end{definition}
\begin{definition}[Strictly increasing]
A function $u:\Re^n \to \Re$ is \emph{strictly increasing} if for \hbox{$y,y' \in \Re^n$}
\[
y < y' \textrm{ implies } u(y) < u(y').
\]
\end{definition}
\begin{definition}[Strongly increasing]
A function $u:\Re^n \to \Re$ is \emph{strongly increasing} if for \hbox{$y,y' \in \Re^n$}
\[
y \leq y' \textrm{ and } y \neq y' \textrm{ implies } u(y) < u(y').
\]
\end{definition}

\section{Deterministic multiobjective optimization}
The main goal of the present paper is to understand which robust efficient solutions can be found as solutions to convex optimization problems. To be able to put our results in context, we first introduce the concept of efficiency for the deterministic case and review a collection of well-known results on the calculation of efficient solutions. 

\subsection{Pareto efficiency}
Deterministic multiobjective problems are formulated in this paper as the minimization of a vector-valued function \hbox{$f: \Re^m \to \Re^n$} over a feasible set \hbox{$\X \subseteq \Re^m$} according to
\begin{equation}\label{mco}
\begin{array}{ll}
\minimize{x \in \X} \quad & \displaystyle \left\{ f(x) := \Big( f_1(x) \cdots  f_n(x) \Big)^T \right\} .
\end{array}
\end{equation}
This formulation is a convex problem if \hbox{$f_1,\ldots,f_n$} are convex functions and $\X$ is a convex set. We understand optimality to formulation~\eqref{mco} in a Pareto sense, meaning that a solution is efficient if it satisfies the following criterion:
\begin{definition}[Pareto efficiency]\label{efficiency}
A feasible solution $x^*$ to problem~\eqref{mco} is \emph{Pareto efficient} (or simply \emph{efficient}) if there is no feasible $x$ such that \[f(x) \in f(x^*) - (\Re^n_+ \setminus \{0\}).\]
\end{definition}
Throughout this paper, we assume that $\X$ is nonempty and compact and that $f$ is lower semicontinuous. These conditions ensure that efficient solutions exist~\citep[Theorem 2.19]{ehrgott05}. 

\subsection{Necessary and sufficient conditions for efficiency}
Efficient solutions to formulation~\eqref{mco} can be found by minimization of a scalarizing function \hbox{$u: \Re^n \to \Re$} over $f(\X)$. The scalarized problem takes the form
\begin{equation}\label{mco-scalarized}
\begin{array}{ll}
\minimize{x \in \X} \quad & \displaystyle u(f(x)).
\end{array}
\end{equation}
The following conditions relate the optimal solutions to the scalarized problem~\eqref{mco-scalarized} to the efficient set of formulation~\eqref{mco}~\citep[Theorems 3.4.5 and 3.5.4]{miettinen99}:
\begin{theorem}[Necessary condition for efficiency]\label{thm-nec-mco}
A solution is efficient to problem~\eqref{mco} only if it is optimal to some scalarized problem according to formulation~\eqref{mco-scalarized} with strictly increasing and convex scalarizing function.
\end{theorem}
\begin{theorem}[Sufficient condition for efficiency]\label{thm-suff-mco}
A solution is efficient to problem~\eqref{mco} if it is optimal to some scalarized problem according to formulation~\eqref{mco-scalarized} with strongly increasing scalarizing function.
\end{theorem}

The necessary conditions can be strengthened to scalarization with strictly increasing and linear functions if the optimization problem is convex~\citep[Theorem 3.1.4]{miettinen99}:
\begin{theorem}[Necessary condition for efficiency for convex problems]\label{thm-nec-convex-mco}
A solution is efficient to a convex multiobjective problem according to formulation~\eqref{mco} only if it is optimal to some scalarized problem according to formulation~\eqref{mco-scalarized} with strictly increasing and linear scalarizing function.
\end{theorem}

Theorems~\ref{thm-nec-mco}--\ref{thm-nec-convex-mco} provide the foundation for algorithms that calculate representations of the efficient set by solving multiple scalarized problems, see, e.g.,~\citet{ruzika05} for a review. If the original multiobjective problem is convex, then solving the scalarized problems requires only convex optimization. 

\section{Robust multiobjective optimization}\label{sec-robust}
In this section, we introduce the concept of robust efficiency, and discuss how robust efficient solutions can be found by scalarization. We then prove necessary and sufficient conditions for robust efficiency that form counterparts of Theorems~\ref{thm-nec-mco} and~\ref{thm-suff-mco} for robust multiobjective optimization.

\subsection{Robust efficiency}
Suppose that formulation~\eqref{mco} is subject to uncertainty. Then, the formulation can be cast as a problem where the objective functions \hbox{$f_1,\ldots,f_n$} depend on some random variable $S$ with range $\S$. We consider a solution to be robust only if it is feasible for every scenario. Thus, the feasible set $\X$ can be assumed not to be parameter-dependent: a constraint \hbox{$x \in \X(s)$} that is required to hold for all $s$ in $\S$ can be posed as the equivalent deterministic constraint \hbox{$x \in \cap_{s \in \S} \X(s)$}. Further, we assume that $\S$ is nonempty and compact, that the lower semicontinuity of \hbox{$f_1(\cdot;s),\ldots,f_n(\cdot;s)$} holds for any $s$ in $\S$, and that \hbox{$f_1(x;\cdot),\ldots,f_n(x;\cdot)$} are continuous for any $x \in \X$. These conditions ensure that the optimization problems that we subsequently consider have well-defined solution sets. To implement robustness against the scenarios in $\S$, the optimization should hedge against the worst case objective value over this set. The robust optimization counterpart of~\eqref{mco} is therefore
\begin{equation}\label{robust-mco}
\begin{array}{ll}
\minimize{x \in \X} \quad & \displaystyle \max_{s \in \S}~\left\{ f(x;s) := \Big( f_1(x;s) \cdots  f_n(x;s) \Big)^T \right\} .
\end{array}
\end{equation}
This formulation is the standard minimax formulation of robust optimization, but with a vector-valued objective function $f$. The difficulty with the formulation is that because $f$ is vector-valued, it constitutes a bi-level multiobjective problem: the inner-level maximization is a multiobjective problem for each fixed value of $x$. The conventional concept of efficiency according to Definition~\ref{efficiency} is therefore not valid for characterizing the efficient set to formulation~\eqref{robust-mco}. A more general concept of robust Pareto efficiency is instead needed. The following definition is due to~\citet{ehrgott14}:

\begin{definition}[Robust Pareto efficiency]\label{robust-efficiency}
A feasible solution $x^*$ to problem~\eqref{robust-mco} is \emph{robust Pareto efficient} (or simply \emph{robust efficient}) if there is no feasible $x$ such that \[f(x;\S) \subseteq f(x^*;\S) - (\Re^n_+ \setminus \{0\}).\]
\end{definition}
Note that this definition is on the same form as the standard definition of efficiency (Definition~\ref{efficiency}), but with set-membership generalized to subset inclusion. The concept of robust efficiency is illustrated in Figure~\ref{fig-re}. 

\begin{figure}[hbtp]
\centering
\includegraphics[width=6.5cm]{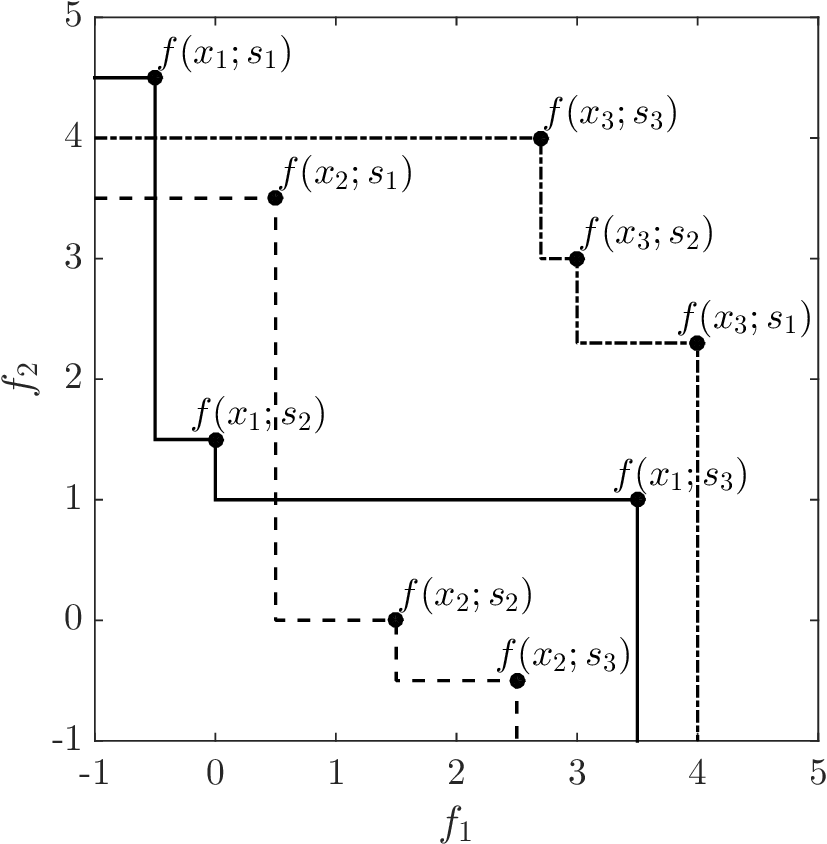}
\caption{Robust efficiency for an uncertain multiobjective problem with two objectives $f_1$ and $f_2$, a feasible set \hbox{$\X = \{ x_1, x_2, x_3 \}$}, and an uncertainty set \hbox{$\S=\{s_1,s_2,s_3\}$}. The solutions $x_1$ (solid) and $x_2$ (dashed) are robust efficient, while $x_3$ (dot-dashed) is dominated by $x_2$ according to Definition~\ref{robust-efficiency}.}
\label{fig-re}
\end{figure} 

\subsection{Necessary and sufficient conditions for robust efficiency}
In this section, we prove counterparts of Theorems~\ref{thm-nec-mco} and~\ref{thm-suff-mco} for robust multiobjective optimization concerning scalarization of~\eqref{robust-mco} with some upper semicontinuous scalarizing function \hbox{$u: \Re^n \to \Re$}. The scalarized problem for the uncertain case takes the form
\begin{equation}\label{robust-mco-scalarized}
\begin{array}{ll}
\minimize{x \in \X} \quad &  \displaystyle \max_{s \in \S} u(f(x;s)),
\end{array}
\end{equation}
where the maximum is attained due to the upper semicontinuity of $u$ and compactness of $f(x;\S)$. It is worth noticing some situations when this formulation constitutes a tractable optimization problem: If $f$ and $\X$ are convex, $\S$ is finite, and $u$ is increasing and convex, then the problem can be tractably solved on its epigraph form:
\begin{equation*}
\begin{array}{lll}
\minimize{x, \lambda} \quad & \lambda \\
\subject & \lambda \geq u(f(x;s)) & \quad s \in \S,\\
& x \in \X.
\end{array}
\end{equation*}
If $\S$ is infinite, there are special cases when~\eqref{robust-mco-scalarized} can be reformulated using robust optimization techniques. For example, if $u$ is a nonnegative vector, \hbox{$\S = \{s \in \Re^{n_s}: A s \leq b, s \geq 0\}$} is polyhedral, and \hbox{$f(x;s) = F(x) s$} is linear in $s$, where \hbox{$F : \Re^m \to \Re^n \times \Re^{n_s}$} is convex and matrix-valued, then strong duality for linear programming can be used to reformulate the problem
\[
\begin{array}{lll}
\minimize{x \in \X} \quad & \displaystyle \max_{s \in \S}  u^T F(x)s
\end{array}
\]
into its tractable equivalent
\[
\begin{array}{lll}
\minimize{x, y} \quad & b^T y \\
\subject & A^T y \geq F(x)^T u, \\
& x \in \X.
\end{array}
\]
For more details on tractable robust counterparts of uncertain optimization problems, see, e.g.,~\citet{ben-tal98},~\citet{bertsimas04}, and~\citet{ben-tal09}.

We now state and prove the necessary condition for robust efficiency:
\begin{theorem}[Necessary condition for robust efficiency]\label{thm-nec-robust-mco}
A solution is robust efficient to problem~\eqref{robust-mco} only if it is optimal to some scalarized problem according to formulation~\eqref{robust-mco-scalarized} with strictly increasing scalarizing function. 
\end{theorem}
\proof~Let $x^*$ be a robust efficient solution to~\eqref{robust-mco} and $u$ be the scalarizing function 
\begin{equation}\label{lp-fun}
u(y) = \displaystyle \min_{z \in f(x^*;\S) - \Re^n_+} \max_{i=1,\ldots,n} y_i-z_i,
\end{equation}
i.e., $u(y)$ measures the signed maximum distance between $y$ and the boundary of \hbox{$f(x^*;\S)-\Re^n_+$}. This function is strictly increasing because
\begin{equation*}
\begin{split}
  y < w  &\Rightarrow  y - z < w - z \quad \forall z \in \Re^n \\
 &\Rightarrow  \max_{i=1,\ldots,n} y_i - z_i < \max_{i=1,\ldots,n} w_i - z_i \quad \forall z \in \Re^n \\
 &\Rightarrow  u(y) < u(w).
\end{split}
\end{equation*}

To see that $x^*$ is optimal to minimization of $u$ over $\X$, observe that
\[
\max_{s \in \S} u(f(x^*;s)) = 0 \leq \max_{s\in\S} u(f(x;s))
\]
for all feasible $x$. The equality is due to the construction of $u$ while the inequality is due to the fact that 
\[
\textrm{$u(y) \geq 0~\forall y \not \in f(x^*;\S) - (\Re^n_+ \setminus \{0\})$}
\]
and that, by robust efficiency of $x^*$, there is no feasible $x$ such that 
\[
f(x;\S) \subseteq f(x^*;\S)-(\Re^n_+ \setminus \{0\}).
\]
Hence, $u$ is a strictly increasing scalarizing function under which $x^*$ is optimal. \hfill \endproof

Theorem~\ref{thm-nec-robust-mco} constitutes a counterpart of Theorem~\ref{thm-nec-mco} for robust multiobjective optimization. It is, however, not a direct parallel of Theorem~\ref{thm-nec-mco} because it is not stated with respect to convex scalarizing functions. The following example shows that a strengthening of the result to convex scalarizing functions is not possible:
\begin{example}\label{ex-rob-is-not-optimal-to-convex}
\emph{There are uncertain convex multiobjective problems with robust efficient solutions that are not optimal to any scalarized problem with strictly increasing and convex scalarizing function.} To see this, we consider the following robust counterpart of an uncertain convex multiobjective problem:
\begin{equation}\label{problem-1}
\minimize{x \in [0,1]} \max_{s = s_1, s_2, s_3} f(x;s), 
\end{equation}
where
\begin{eqnarray*}
f(x;s_1) &=& x ( 0 ~ 2 )^T + (1-x) ( 1 ~ 4 )^T,\\
f(x;s_2) &=& x ( 2 ~ 2 )^T + (1-x) ( 1 ~ 1 )^T,\\
f(x;s_3) &=& x ( 2 ~ 0 )^T + (1-x) ( 4 ~ 1 )^T.
\end{eqnarray*}
The image under the objective function mapping of the extreme points of the feasible set of this problem is illustrated in Figure~\ref{fig-not-che}.
\begin{figure}[hbtp!]
\centering
\includegraphics[width=6.5cm]{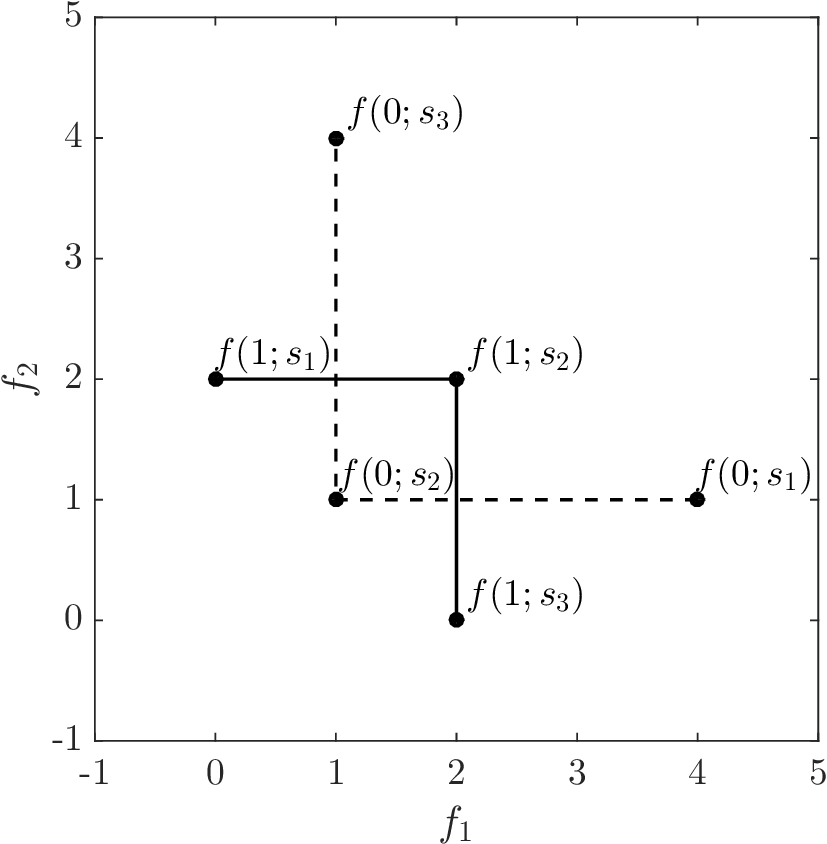}
\caption{Image under the objective function mapping of the extreme points of the feasible set of formulation~\eqref{problem-1}. The robust efficient solution $x=0$ (dashed) is not optimal to any strictly increasing and convex scalarizing function.}\label{fig-not-che}
\end{figure}

Because $f(x;s_2) > f(0;s_2)$ for any $x \in (0,1]$, there is no $x \in [0,1]$ such that \hbox{$f(x;\S) \subseteq f(0;\S) - (\Re_+^n \setminus \{0\})$}, so the solution $x=0$ is robust efficient. Under scalarization with any strictly increasing and convex scalarizing function \hbox{$u : \Re^2 \to \Re$}, however, the solution $x=0$ is dominated by $x=1$, i.e.,
\begin{equation}\label{ex1}
  \max_{s=s_1,s_2,s_3} u(f(1;s)) < \max_{s=s_1,s_2,s_3} u(f(0;s)).
\end{equation}
To see this, we observe that because $u$ is increasing, it holds that
  \[
  \max_{s=s_1,s_2,s_3} u(f(1;s)) = u(f(1;s_2))
  \]
  and
  \[
  \max_{s=s_1,s_2,s_3} u(f(0;s)) =  \max_{s=s_1,s_3} u(f(0;s)).
  \]
Hence, if
\begin{equation}\label{ex1-2}
u(f(1;s_2)) < \max_{s=s_1,s_3} u(f(0;s)),
\end{equation}
then~\eqref{ex1} follows. Because \hbox{$f(1;s_2) = (2~2)^T$}, \hbox{$f(0;s_1) = (1~4)^T$}, and \hbox{$f(0;s_3) = (4~1)^T$}, the inequality~\eqref{ex1-2} can be derived in the following manner:
\begin{eqnarray*}
u\left( ( 2 ~ 2 )^T \right) &< & u\left( \frac{1}{2} ( 1 ~ 4 )^T + \frac{1}{2} ( 4 ~ 1 )^T \right)\\
&\leq& \frac{1}{2} u\left( ( 1 ~ 4 )^T \right) + \frac{1}{2} u\left( ( 4 ~ 1 )^T \right)\\
&\leq& \max_{\alpha \in [0,1]} \left\{ \alpha u\left( ( 1 ~ 4 )^T \right) + (1 -\alpha) u\left( ( 4 ~ 1 )^T \right) \right\}\\
&=& \max \left\{ u\left( ( 1 ~ 4 )^T \right), u\left( ( 4 ~ 1 )^T \right) \right\},
\end{eqnarray*}
where the inequalities are respectively due to that $u$ is strictly increasing, that $u$ is convex (Jensen's inequality), and that \hbox{$1/2 \in [0,1]$}. The equality is due to that the maximum over all convex combinations of two numbers is attained for one of the two numbers. This shows that there are convex problems with robust efficient solutions that are not optimal under any strictly increasing and convex scalarization.
\end{example}

In order to extend Theorem~\ref{thm-suff-mco} to robust efficiency, we first prove the following lemma, which says that if one feasible solution dominates another in Pareto sense, then the dominance relation also holds under increasing scalarizations:
\begin{lemma}\label{strict-ineq-prim}
For two feasible solutions $x$ and $x'$ to problem~\eqref{robust-mco}, it holds that if 
\[
 f(x'; \S) \subseteq f(x;\S) - (\Re^n_+ \setminus \{0\}),
 \] 
then
\[
\max_{s \in \S} u( f(x'; s)) \leq \max_{s \in \S} u( f(x; s)) 
\]
for all upper semicontinuous and increasing functions \hbox{$u: \Re^n \to \Re$}. If $u$ is strongly increasing, then the result holds with strict inequality.
\end{lemma}

\proof~Assume that the feasible solutions $x$ and $x'$ satisfy
\[ 
f(x'; \S) \subseteq f(x;\S) - (\Re^n_+ \setminus \{0\}). 
\]
Then for any \hbox{$s' \in \S$}, there exists an \hbox{$s \in \S$} and $\lambda \in \Re_+^n\setminus\{0\}$ such that \hbox{$f(x';s') = f(x;s) - \lambda$}. Hence, for any increasing scalarizing function $u$, it holds that
\begin{equation*}
  \forall s' \in \S~\exists s \in \S, \lambda \in \Re_+^n\setminus\{0\} : u(f(x';s')) = u(f(x;s)-\lambda) \leq u(f(x;s)),
\end{equation*}
where the inequality is due to that $u$ is increasing. This implies that
\begin{equation*}
 \max_{s \in \S} u(f(x';s)) \leq \max_{s \in \S} u(f(x;s)).
\end{equation*}
If $u$ is strongly increasing, the inequalities are strict. \hfill \endproof

The following counterpart of Theorem~\ref{thm-suff-mco} for robust multiobjective optimization is a direct corollary of Lemma~\ref{strict-ineq-prim}:
\begin{theorem}[Sufficient condition for robust efficiency]\label{thm-suff-robust-mco}
A solution is robust efficient to problem~\eqref{robust-mco} if it is optimal to some scalarized problem according to formulation~\eqref{robust-mco-scalarized} with strongly increasing scalarizing function.
\end{theorem}
A version of this theorem specialized to strongly increasing and linear scalarizations has previously been shown by~\citet[Theorem 4.3]{ehrgott14}.

\section{Convex robust multiobjective optimization}\label{sec-convex-hull-efficiency}
Because of the superior tractability of convex optimization, the scalarizing functions that are used in practice are often convex. Example~\ref{ex-rob-is-not-optimal-to-convex} shows that there are robust efficient solutions that are not optimal under any strictly increasing and convex scalarizing function. This fact motivates the characterization of the subset of robust efficient solutions that are optimal under such scalarization. To this end, we introduce the concept of \emph{convex hull efficiency}, and prove counterparts of Theorems~\ref{thm-nec-mco} and~\ref{thm-suff-mco} for convex scalarizing functions applied to uncertain multiobjective problems.

\subsection{Convex hull efficiency}
We use the connection between minimax robustness and probabilistic optimization to define convex hull efficiency. If $f$ is scalar-valued, then problem~\eqref{robust-mco} is equivalent to the following:
\begin{equation}\label{robust-exp}
   \begin{array}{ll}
     \minimize{x \in \X} \quad & \displaystyle \max_{\pi \in \allP}~\Ex_{\pi}[f(x;S)],
   \end{array}
\end{equation}
where $\allP$ is is the set of all probability distributions over $\S$, see, e.g.,~\citet{shapiro02}. If $f$ is vector-valued, then problems~\eqref{robust-mco} and~\eqref{robust-exp} do not generally have the same robust efficient sets. The robust efficient solutions to~\eqref{robust-exp} satisfy the following more restrictive definition of robustness with respect to~\eqref{robust-mco}:
\begin{definition}[Convex hull robust Pareto efficiency]\label{convex-hull-efficiency}
A feasible solution $x^*$ to problem~\eqref{robust-mco} is \emph{convex hull robust Pareto efficient} (or simply \emph{convex hull efficient}) if there is no feasible $x$ such that \[f(x;\S) \subseteq \conv(f(x^*;\S)) - (\Re^n_+ \setminus \{0\}).\]
\end{definition}

Below, we formally show that the robust efficient solutions to formulation~\eqref{robust-exp} and the convex hull efficient solutions to formulation~\eqref{robust-mco} coincide. We also show that the convex hull efficient solutions constitute a proper subset of the robust efficient solutions in the general case, and that the two sets coincide under some special circumstances. The concept of convex hull efficiency is illustrated in Figure~\ref{fig-che}.

\begin{figure}[hbtp]
\centering
\includegraphics[width=6.5cm]{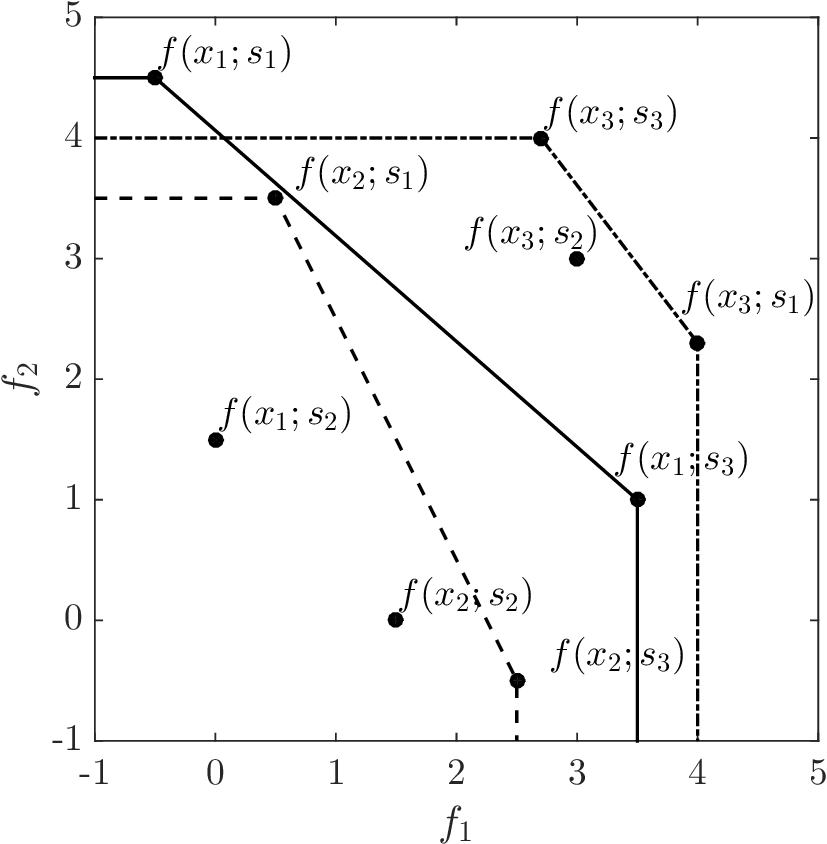}
\caption{Convex hull efficiency for an uncertain multiobjective problem with two objectives $f_1$ and $f_2$, a feasible set \hbox{$\X = \{ x_1, x_2, x_3 \}$}, and an uncertainty set \hbox{$\S=\{s_1,s_2,s_3\}$}. The solution $x_2$ (dashed) is convex hull efficient, while $x_1$ (solid) and $x_3$ (dot-dashed) are dominated by $x_2$ according to Definition~\ref{convex-hull-efficiency}. Note that the set of convex hull efficient solutions is different from the set of robust efficient solutions (compare to Figure~\ref{fig-re}).}
\label{fig-che}
\end{figure}

\begin{proposition}\label{che-exp}
A solution is convex hull efficient to problem~\eqref{robust-mco} if and only if it is robust efficient to problem~\eqref{robust-exp}.
\end{proposition}

\proof~We first show that all robust efficient solutions to~\eqref{robust-exp} are convex hull efficient to~\eqref{robust-mco}. Let \hbox{$g(x;\pi) = \Ex_\pi [ f(x;S) ]$}, so that~\eqref{robust-exp} can be put on the form of~\eqref{robust-mco} according to
\[
\min_{x \in \X} \max_{\pi \in {\allP}} g(x;\pi).
\]
For any robust efficient solution $x^*$ to this problem, there is by Definition~\ref{robust-efficiency} no feasible $x$ such that
\[
g(x;{\allP}) \subseteq g(x^*;{\allP}) - (\Re^n_+ \setminus \{0\}).
\]
We use that $\allP = \conv(\Dirac)$, where $\Dirac$ is the set of all Dirac distributions $\delta(s)$, assigning unit probability to the random variable $S$ taking on the value $s$, for $s$ in $\S$. Then, because
\begin{equation*}
\begin{split}
g(x;\allP) &= \left\{ \Ex_\pi [ f(x;S) ]: \pi \in \allP \right\} \\
&= \left\{ \sum_{i \in I} \lambda_i \Ex_{\delta(s_i)} [ f(x;S) ]: \delta(s_i) \in \Dirac,~ \lambda_i > 0,~ \sum_{i \in I} \lambda_i =1,~ |I| < \infty \right\} \\
&= \left\{ \sum_{i \in I} \lambda_i f(x;s_i): s_i \in \S,~ \lambda_i > 0,~ \sum_{i \in I} \lambda_i =1,~ |I| < \infty \right\} \\
&= \conv(f(x;\S)),
\end{split}
\end{equation*}
there is no feasible $x$ such that
\begin{equation}\label{conv-subseteq}
\conv(f(x;\S)) \subseteq \conv(f(x^*;\S)) - (\Re^n_+ \setminus \{0\}). 
\end{equation}
Therefore, there is no feasible $x$ such that
\begin{equation}\label{f-subseteq}
f(x;\S) \subseteq \conv(f(x^*;\S)) - (\Re^n_+ \setminus \{0\}).
\end{equation}
Hence, $x^*$ is convex hull efficient to~\eqref{robust-mco}. Taking these steps backwards shows the converse. Note that~\eqref{f-subseteq} yields~\eqref{conv-subseteq} because the right-hand side of~\eqref{f-subseteq} is a convex set and the convex hull of a set is a subset of all convex sets that contain the set. \hfill \endproof

\begin{proposition}\label{che-is-efficient}
A solution is convex hull efficient to problem~\eqref{robust-mco} only if it is robust efficient to the same problem.
\end{proposition}

\proof~Let $x^*$ be convex hull efficient. Then, there exists no feasible $x$ such that
\[
f(x;\S) \subseteq \conv(f(x^*;\S)) - (\Re^n_+ \setminus \{0\}),
\] 
which combined with the fact that 
\[
 f(x^*;\S) - (\Re^n_+ \setminus \{0\}) \subseteq \conv(f(x^*;\S)) - (\Re^n_+ \setminus \{0\})
\] 
yields that there exists no feasible $x$ such that
\[
f(x;\S) \subseteq f(x^*;\S) - (\Re^n_+ \setminus \{0\}),
\] 
and, therefore, $x^*$ is robust efficient. \hfill \endproof

A comparison between Figures~\ref{fig-re} and~\ref{fig-che} shows that there are robust efficient solutions that are not convex hull efficient. The following example strengthens this result to convex problems.
\begin{example}\label{ex-efficient-is-not-che}
\emph{There are uncertain convex multiobjective problems with robust efficient solutions that are not convex hull efficient.} To see this, we consider the problem~\eqref{problem-1} from Example~\ref{ex-rob-is-not-optimal-to-convex}. The solution \hbox{$x = 0$} is robust efficient because there exists no feasible $x$ such that 
\[
f(x;\S) \subseteq f(0;\S) - (\Re^n_+ \setminus \{0\}). 
\]
The solution \hbox{$x=1$}, however, satisfies
\[ 
f(1;\S) \subseteq \conv(f(0;\S)) - (\Re^n_+ \setminus \{0\}),
\]
which shows that \hbox{$x=0$} is not convex hull efficient.
\end{example}

However, there are also situations in which convex hull efficiency and robust efficiency coincide:
\begin{example}\label{remark}
\emph{Under some conditions, the robust efficient solutions and the set of convex hull efficient solutions coincide.} This can be exemplified by any of the following conditions:
\begin{itemize}
\item[(a)] $n=1$;
\item[(b)] $\S$ is a singleton set; or
\item[(c)] $f$ is a linear function of $s$ and $\S$ is a convex set.
\end{itemize}
This holds true because
\begin{equation}\label{eq-che-is-e}
\conv(f(x;\S)) - (\Re^n_+ \setminus \{0\}) = f(x;\S) - (\Re^n_+ \setminus \{0\})
\end{equation}
for any feasible $x$ under any of the stated conditions, as shown as follows:
\begin{itemize}
\item[(a)]
If $n=1$, then~\eqref{eq-che-is-e} holds for any feasible $x$ because
\[
\{ a - r : a \in \conv( A ),~r \in \Re,~r > 0 \} = \{ a - r : a \in A,~r \in \Re,~r > 0 \}
\]
for any \hbox{$A \subseteq \Re$}.

\item[(b)]
If \hbox{$\S=\{s\}$}, then~\eqref{eq-che-is-e} holds for any feasible $x$ because the convex hull of a point is the point itself. 

\item[(c)] 
If $f$ is linear in $s$ and $\S$ is convex, then~\eqref{eq-che-is-e} holds for any feasible $x$ because the image of a convex set under a linear mapping is convex, and the convex hull of a convex set is the set itself. 
\end{itemize}
\end{example}

\subsection{Necessary and sufficient conditions for convex hull efficiency}
In this section, we prove counterparts of Theorems~\ref{thm-nec-mco} and~\ref{thm-suff-mco} for robust multiobjective optimization that concerns scalarization of~\eqref{robust-mco} with some convex \hbox{$u: \Re^n \to \Re$} according to~\eqref{robust-mco-scalarized}. We first state and show the necessary condition for convex hull efficiency:

\begin{theorem}[Necessary condition for convex hull efficiency]\label{thm-nec-che-robust-mco}
A solution is convex hull efficient to problem~\eqref{robust-mco} only if it is optimal to some scalarized problem according to formulation~\eqref{robust-mco-scalarized} with strictly increasing and convex scalarizing function. 
\end{theorem}

\proof~The proof is a direct analog of the proof of Theorem~\ref{thm-nec-robust-mco}, but with $\conv(f(x^*;\S))$ substituted for $f(x^*;\S)$ and the addition of a proof that $u$ according to~\eqref{lp-fun} is a convex function. Convexity of $u$ follows from the fact that this function is the optimal value function to a minimization problem where $z$ is minimized over a convex set independent from $y$, $y$ is restricted to a convex set (viz.,~$\Re^n$), and the inner-level maximum in~\eqref{lp-fun} is a jointly convex function in $(y,z)$~\citep[Proposition 2.1]{fiacco86}. 
\hfill \endproof

In order to extend Theorem~\ref{thm-suff-mco} to convex hull efficiency, a lemma that is analogous to Lemma~\ref{strict-ineq-prim} is needed.
\begin{lemma}\label{strict-ineq}
For two feasible solutions $x$ and $x'$ to problem~\eqref{robust-mco}, it holds that if 
\[
 f(x'; \S) \subseteq \conv(f(x;\S)) - (\Re^n_+ \setminus \{0\}),
 \] 
then
\[
\max_{s \in \S} u( f(x'; s)) \leq \max_{s \in \S} u\left( f(x; s) \right),
\]
for all upper semicontinuous and increasing and convex functions \hbox{$u: \Re^n \to \Re$}. If $u$ is strongly increasing and convex, then the result holds with strict inequality.
\end{lemma}

\proof~Assume that the feasible solutions $x$ and $x'$ satisfy
\[ 
f(x'; \S) \subseteq \conv( f(x;\S) ) - (\Re^n_+ \setminus \{0\}). 
\]
Then for any \hbox{$s' \in \S$}, there exists a probability distribution $\pi \in \allP$ and \hbox{$\lambda \in \Re_+^n\setminus\{0\}$} such that \hbox{$f(x';s') = \Ex_\pi[f(x;S)] - \lambda$}. Hence, for any increasing scalarizing function $u$, it holds that
\begin{equation*}
  \forall s' \in \S~\exists \pi \in \allP, \lambda \in \Re_+^n\setminus\{0\} : u(f(x';s')) = u\left(\Ex_\pi\left[f(x;S)\right]-\lambda\right) \leq u\left(\Ex_\pi\left[f(x;s)\right]\right),
\end{equation*}
which implies that
\begin{equation*}
 \max_{s' \in \S} u(f(x';s')) \leq \max_{\pi \in \allP} u\left(\Ex_\pi\left[f(x;s)\right]\right).
\end{equation*}
The inequalities are strict for strongly increasing $u$. If furthermore $u$ is convex, we have
\begin{equation}\label{jensen}
  \max_{\pi \in \allP} u\left( \Ex_\pi[ f( x; S) ] \right) \leq \max_{ \pi \in \allP}  \Ex_\pi[ u( f(x; S) )] =
  \max_{s \in \S} u( f(x; s) ),
\end{equation}
where the inequality is due to Jensen's inequality and the equality is due to linearity of the expectation operator. \hfill \endproof

The following counterpart of Theorem~\ref{thm-suff-mco} for convex hull efficiency is a direct corollary of Lemma~\ref{strict-ineq}:
\begin{theorem}[Sufficient condition for convex hull efficiency]\label{thm-suff-che-robust-mco}
A solution is convex hull efficient to problem~\eqref{robust-mco} if it is optimal to some scalarized problem according to formulation~\eqref{robust-mco-scalarized} with strongly increasing and convex scalarizing function.
\end{theorem}

The following example shows that Theorem~\ref{thm-nec-convex-mco} does not extend to convex hull efficiency, and hence neither to robust efficiency.
\begin{example}\label{ex-che-is-not-optimal-to-linear}
\emph{There are uncertain convex multiobjective problems with convex hull efficient solutions that are not optimal to any scalarized problem with strictly increasing and linear scalarizing function.} This is exemplified by the following robust counterpart of an uncertain convex multiobjective problem:
\begin{equation}\label{problem-2}
\begin{array}{lll} 
\minimize{x \in \Re^2} ~\quad & \displaystyle \max_{s = s_1,s_2,s_3} f(x;s) & \\
\subject & x_1 + x_2 \leq 1, & \\
         & x_i \geq 0, & i=1,2,
\end{array}
\end{equation}
where
\begin{eqnarray*}
f(x;s_1) &=& x_1 ( 0 ~ 6 )^T + x_2 ( 3 ~ 5 / 2 )^T + (1-x_1-x_2) ( 2 ~ 4 )^T, \\
f(x;s_2) &=& x_1 ( 0 ~ 3 )^T + x_2 ( 3 ~ 0 )^T + (1-x_1-x_2) ( 4 ~ 4 )^T, \\
f(x;s_3) &=& x_1 ( 5 / 2 ~ 3 )^T + x_2 ( 6 ~ 0 )^T + (1-x_1-x_2) ( 4 ~ 2 )^T. 
\end{eqnarray*}
The image under the objective function mapping of the extreme points of the feasible set of this problem is illustrated in Figure~\ref{fig-not-lin-opt}.
\begin{figure}[hbtp!]
\centering
\includegraphics[width=6.5cm]{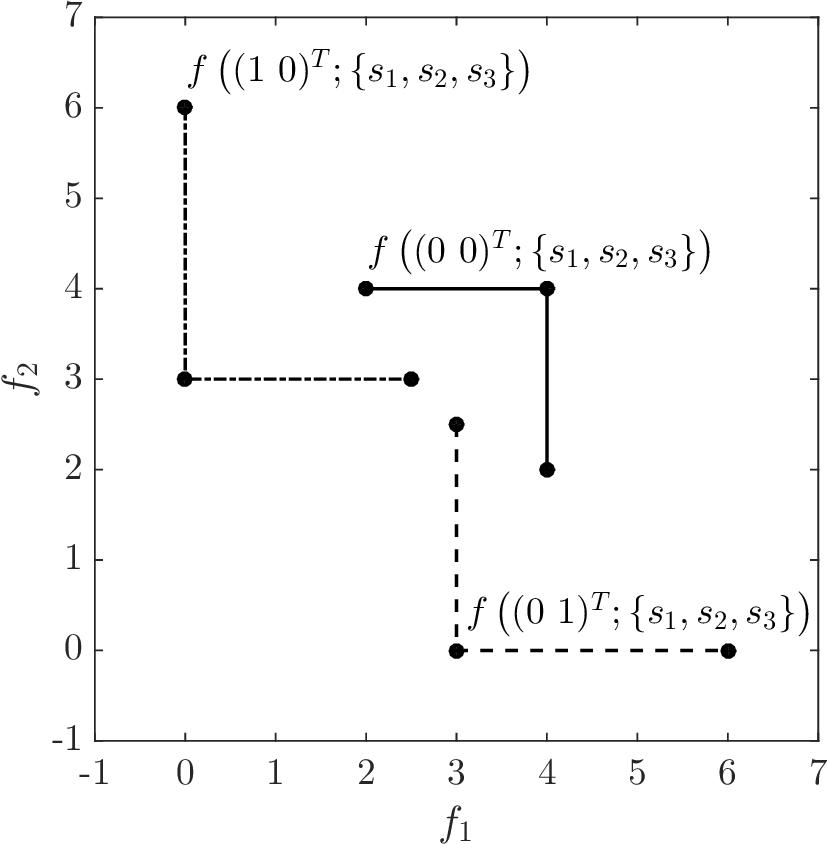}
\caption{Image under the objective function mapping of the extreme points of the feasible set of formulation~\eqref{problem-2}. The convex hull efficient solution \hbox{$x=(0~0)^T$} (solid) is not optimal to any strictly increasing and linear scalarization.}
\label{fig-not-lin-opt}
\end{figure}
The solution \hbox{$x=(0~0)^T$} is convex hull efficient because there exists no feasible solution that is entirely contained in the set \hbox{$\conv(f((0~0)^T;\S)) - (\Re^n_+ \setminus \{0\})$}. Now consider scalarization with an arbitrary strictly increasing and linear function, i.e., \hbox{$u(y) = w_1 y_1 + w_2 y_2$} for some $w$ in \hbox{$\Re_+^2 \setminus \{0\}$}. Without loss of generality, we assume that \hbox{$w_1 + w_2 = 1$} and, by symmetry of the problem, that \hbox{$w_1 \leq w_2$}. These assumptions yield that
\begin{equation*}
\max_{s = s_1,s_2,s_3} u(f((0~1)^T;s)) =  \max \Big\{3w_1+5/2w_2,3w_1,6w_1 \Big\} \leq 3, 
\end{equation*}
and
\begin{equation*}
\max_{s = s_1,s_2,s_3} u(f((0~0)^T;s)) =  \max \Big\{2w_1+4w_2,4w_1+4w_2,4w_1+2w_2 \Big\} = 4.
\end{equation*}
This shows that \hbox{$x=(0~1)^T$} (or \hbox{$x=(1~0)^T$}) has a strictly smaller scalarized objective function value than \hbox{$x=(0~0)^T$} with respect to any strictly increasing and linear scalarizing function, and thereby that there are convex problems with convex hull efficient solutions that are not optimal under any strictly increasing and linear scalarization.
\end{example}

We observe that the difference between the necessary conditions for robust efficient solutions (Theorem~\ref{thm-nec-robust-mco}) and those for convex hull efficient solutions (Theorem~\ref{thm-nec-che-robust-mco}) is that for the latter, there is a convex scalarizing function. This observation naturally leads to the question of whether necessary conditions that guarantee a convex scalarizing function can be formulated with respect to a larger subset of the robust efficient solutions. To answer this question, we consider a more general concept of efficiency defined by an arbitrary mapping between subsets of $\Re^n$ according to:
\begin{definition}[$\A$-efficiency]\label{a-efficiency}
Given a mapping \hbox{$\A:2^{\Re^n} \rightarrow 2^{\Re^n}$}, a feasible solution $x^*$ to problem~\eqref{robust-mco} is \emph{$\A$-Pareto efficient} (or simply \emph{$\A$-efficient}) if there is no feasible $x$ such that \[f(x;\S) \subseteq \A( f(x^*;\S) ) - (\Re^n_+ \setminus \{0\}).\]
\end{definition}
One example of a mapping of this type is the convex hull operator $\conv(\cdot)$, which induces a dominance relation according to Definition~\ref{convex-hull-efficiency}. Another example is $\{ \sup(\cdot) \}$ (with the supremum taken componentwise), which induces objectivewise efficiency on the form suggested by~\citet{kuroiwa12}. The following proposition shows that there is no compact-preserving mapping $\A$ (i.e., $\A$ maps compact sets to compact sets) that defines a larger subset of the robust efficient solutions than the convex hull efficient solutions and simultaneously allows the necessary conditions to be stated with respect to convex scalarizing functions:

\begin{proposition}\label{prop-che-is-largest-class}
Let \hbox{$\A : 2^{\Re^n} \rightarrow 2^{\Re^n}$} be a compact-preserving mapping such that
\begin{enumerate}[(i)]
\item A solution is $\A$-efficient to problem~\eqref{robust-mco} if it is convex hull efficient to the same problem.
\item A solution is $\A$-efficient to problem~\eqref{robust-mco} only if it is optimal to some scalarized problem according to formulation~\eqref{robust-mco-scalarized} with strictly increasing and convex scalarizing function.
\end{enumerate}
Then $\A$-efficiency is equivalent to convex hull efficiency.
\end{proposition}

\proof~Let \hbox{$\A : 2^{\Re^n} \rightarrow 2^{\Re^n}$} be any compact-preserving mapping such that conditions (i) and (ii) hold. We first show that for any nonempty compact subset $Y$ of $\Re^{n}$, it holds that
\begin{equation}\label{subset-relation}
\A(Y) - (\Re^n_+\setminus\{0\})  \subseteq \conv(Y) - (\Re^n_+\setminus\{0\}).
\end{equation}
Let \hbox{$p \in \A(Y) - (\Re^n_+ \setminus \{0 \})$} and consider the uncertain multiobjective problem with \hbox{$\X = \{x_1,x_2\}$}, \hbox{$\S = Y$}, and \hbox{$f = f_p$}, where
\begin{equation}\label{a-efficient-function}
f_p(x;s) = \left\{ \begin{array}{ll}s \quad \textrm{ if $x = x_1$} \\ p \quad \textrm{ otherwise} \end{array}  \right.
\end{equation}
(i.e., $f$ is defined such that \hbox{$f(x_1;\S) = Y$} and \hbox{$f(x_2;\S) = \{p\}$}). Because $x_1$ is not $\A$-efficient to this problem, (i) implies that it is neither convex hull efficient, so \hbox{$p \in \conv(Y) - (\Re^n_+ \setminus\{0\})$}.

We now show that for any nonempty compact subset $Y$ of $\Re^{n}$, it holds that
\begin{equation}\label{subset-relation2}
 \mathrm{int}(\conv(Y) - (\Re^n_+\setminus\{0\})) \subseteq \A(Y) - (\Re^n_+\setminus\{0\}) .
\end{equation}
Let \hbox{$q \in \mathrm{int}(\conv(Y) - (\Re^n_+ \setminus \{0\}))$}. Then there exists some \hbox{$r \in \conv(Y) - (\Re^n_+ \setminus \{0\})$} such that \hbox{$q < r$}, and hence, for some positive integer $k$, there exist some \hbox{$y_1,\ldots, y_k \in Y$} and \hbox{$\lambda_1,\ldots,\lambda_k \in [0,1]$} with \hbox{$\sum_{i=1}^k \lambda_i = 1$} such that \hbox{$r \leq \sum_{i=1}^k \lambda_i y_i$}. Because $q < r \leq \sum_{i=1}^k \lambda_i y_i$ implies \hbox{$u(q) < \max_{y\in Y} u(y)$} for any strictly increasing, convex, and upper semicontinuous function $u$, there is no such function $u$ for which $x_1$ is optimal to
\[
\minimize{x = x_1,x_2 } \quad \displaystyle \max_{s \in Y}  u(f_q(x;s)),
\]
where $f_q$ is defined according to~\eqref{a-efficient-function} but with $q$ substituted for $p$. Hence, by (ii), $x_1$ is not $\A$-efficient to the uncertain multiobjective problem with \hbox{$\X=\{x_1, x_2\}$}, \hbox{$\S = Y$}, and \hbox{$f = f_q$}, which implies that \hbox{$q \in \A(Y) - (\Re^n_+ \setminus\{0\})$}, and thus~\eqref{subset-relation2}.

The relations~\eqref{subset-relation} and~\eqref{subset-relation2} imply that
\[
\interior(\conv(Y) - (\Re^n_+\setminus\{0\})) = \interior(\A(Y) - (\Re^n_+\setminus\{0\}) ),
\]
which, together with the fact that \hbox{$\A(Y)- \Re^n_+$} and \hbox{$\conv(Y) - \Re^n_+$} are closed, connected, and full-dimensional, yields
\[
\begin{split}
  \conv(Y) - \Re^n_+ 
  &= \closure( \interior ( \conv(Y) - (\Re^n_+\setminus\{0\}) ) )\\ 
  &= \closure( \interior ( \A(Y) - (\Re^n_+\setminus\{0\}) ) )\\
  &= \A(Y) - \Re^n_+.
\end{split}
\]
The equality still holds when the origin is removed from the subtrahends:
\[
\conv(Y) - (\Re^n_+\setminus\{0\}) = \A(Y) - (\Re^n_+\setminus\{0\}) .
\]
\hfill \endproof

\section{Numerical example}\label{sec-numerical-example}
Example~\ref{ex-che-is-not-optimal-to-linear} shows that linear scalarizing functions cannot find all convex hull efficient solutions of general uncertain convex multiobjective problems. This property constitutes an important difference to the deterministic case (cf.~Theorem~\ref{thm-nec-convex-mco}) because it shows that the choice of scalarizing function dictates which parts of the convex hull efficient set that can be found, even for convex problems. Guided by this result, we hypothesize that there are considerable differences between the optimal solutions found by different scalarizing functions and, therefore, that it is important to choose a scalarization that is in line with one's preferences. To examine this hypothesis, we investigated the qualitative behavior of different scalarizations on a test problem by minimization of a weighted $p$-norm 
\begin{equation}\label{u}
u(z) = \left( \frac{1}{n}\sum_{i=1}^n w_i |z_i|^p \right)^{1/p},
\end{equation}
for different vectors of weights $w$ and exponents $p$. We then studied how the choice of $p$ influenced the optimized solutions. The scalarizing function~\eqref{u} is strongly increasing and convex if \hbox{$w \in \Re^n_+ \setminus \{0\}$} and \hbox{$1 \leq p < \infty$}, and thereby, by Theorem~\ref{thm-suff-mco}, finds convex hull efficient solutions under these circumstances. With slight abuse of terminology, we henceforth refer to~\eqref{u} simply as a $p$-norm. Note that optimization with this scalarization approaches the objectivewise formulation of~\citet{kuroiwa12} when \hbox{$p \to \infty$}. We first outline the studied test problem and then summarize our numerical results.

\subsection{Test problem}\label{sec-test-problem}
The problem concerns optimization of a proton therapy treatment for a prostate cancer patient. The objectives used in proton therapy typically concern different anatomic regions. Because the uncertainty is geometric, it affects the objectives in a highly correlated manner. This property makes the choice of scalarizing function important, as the results below show.

In proton therapy, a treatment fraction is delivered by a proton beam that is scanned point-by-point over the tumor volume while the beam energy is adjusted in order to control the penetration depth of the protons, see Figure~\ref{fig-patient}. Proton therapy optimization is inherently multiobjective because the risk of an insufficient dose to the cancer must be balanced against the danger of radiation-induced damage to healthy organs~\citep{bortfeld99}. The optimization must simultaneously incorporate robustness against errors that can occur during treatment such as incorrect patient positioning or organ movement~\citep{engelsman13}. Minimax robustness and multiobjective optimization are on the verge of widespread clinical implementation to handle these issues~\citep{baumann16} (but not yet in combination).

To cast the test problem on the form of~\eqref{robust-mco}, we let $x$ represent the scanning time per position and energy (\hbox{$24612$} nonnegatively constrained variables in total) and introduce two objectives $f_1$ and $f_2$:
\begin{equation}\label{objectives}
\begin{array}{ll}
f_1(x;s) &= \displaystyle \int_{v \in \T} (d(v;s)^T x - \dhat)^2 \, dv, \\
f_2(x;s) &= \displaystyle \int_{v \in \R} (d(v;s)^T x)^2 \, dv + 10^{-2} \int_{v \in \U} (d(v;s)^T x)^2 \, dv.
\end{array}
\end{equation}
Here, $f_1$ penalizes deviation from the dose prescription $\dhat$ inside the tumor volume $\T$. The first term of $f_2$ penalizes dose to the rectal volume $\R$ while the second low-weighted term penalizes dose to the unclassified tissue $\U$. The function $d(v;s)$ is a dose kernel such that $d(v;s)^T x$ is the absorbed dose in some point $v$ under scenario $s$ if irradiated with the scanning times $x$. The problem is subject to uncertainty of the form of rigid shifts of the treatment fields, represented by a set $\S$ composed of 27 scenarios: The nominal scenario and all shifts of $1$ cm in the positive and negative axis directions (\hbox{$(\pm 1~0~0)^T$}, \hbox{$(0~\pm 1 ~ 0)^T$}, \hbox{$(0~0~\pm1)^T$}), in all pairwise combinations of axis directions (\hbox{$(\pm 1 ~ \pm 1 ~ 0)^T$}, \hbox{$(0~\pm 1~\pm 1)^T$}, \hbox{$(\pm 1 ~0 \pm 1 )^T$}), and in all three-way combinations of axis directions (\hbox{$(\pm 1~\pm 1~\pm 1 )^T$}).
\begin{figure}[hbtp]
\centering
\includegraphics[width=13cm]{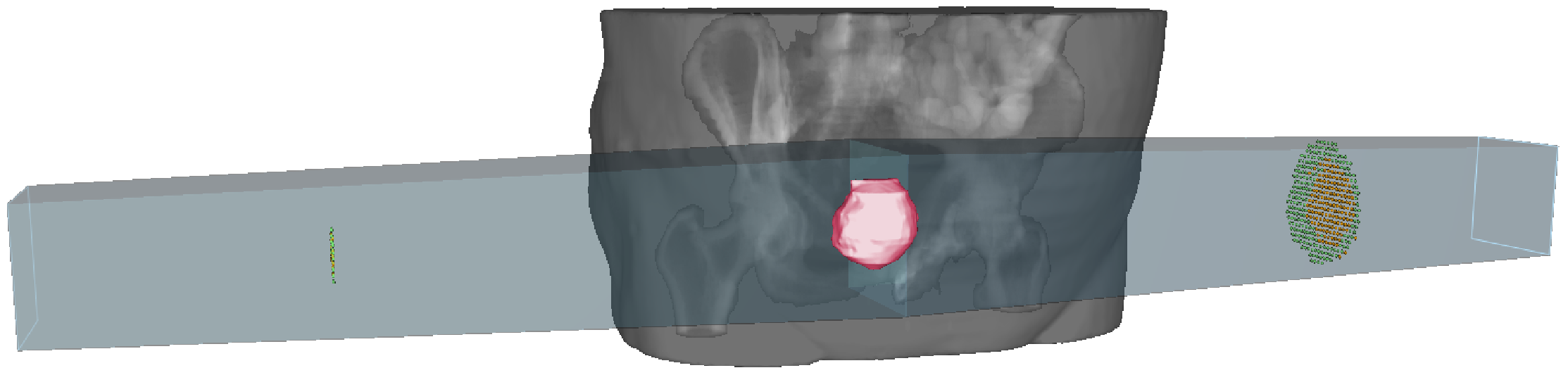}
\caption{Irradiation of a prostate tumor (red) using two parallel-opposed proton fields. The point patterns in the cross-sections of the fields indicate scanning positions.}
\label{fig-patient}
\end{figure}

\subsection{Numerical results}
The test problem was solved in the treatment planning system RayStation v4.0 (RaySearch Laboratories, Stockholm, Sweden), which was modified to permit optimization on the form of~\eqref{robust-mco-scalarized}. We repeated the optimization for \hbox{$p=1,2,$} and $10$ while also varying the weights $w$ in equidistant steps. These values of $p$ were chosen because \hbox{$p=1$} is the smallest parameter value that gives a convex scalarizing function and corresponds to weighted sum minimization, \hbox{$p=2$} is an intermediate value that gives the Euclidean norm, and \hbox{$p=10$} is fairly close to the maximum norm and therefore similar to weighted Tchebycheff scalarization~\citep[see][Equation 2.1.2]{miettinen99}. The parameter values \hbox{$p=1,2,$} and $\infty$ are the common ones in $p$-norm scalarization according to~\citet[Page 68]{miettinen99}. The results for other values between 1 and 10 behave in the manner that can be interpolated from \hbox{$p=1,2,$} and $10$. The integrals in~\eqref{objectives} were evaluated by a discretization of the patient volume into \hbox{$3 \times 3 \times 3$ $\textrm{mm}^3$} volume elements (\hbox{$ \sim 10^6$} elements in total).

Figure~\ref{fig-fvals} depicts the sets \hbox{$\conv(f(x;\S)) - (\Re^n_+ \setminus \{ 0 \})$} associated with the convex hull efficient solutions $x$ found for different weights $w$ and exponents $p$. The sets $f(x;\S)$ that are optimal with respect to an equiweighted scalarizing function $u$ are depicted in greater detail in the lower right subfigure alongside level curves for $u$. Observe that for all weights, the optimal solutions to \hbox{$p=1$} have at most one objective taking on a poor value in any given scenario. In contrast, the optimal solutions to \hbox{$p=2$} and, in particular, \hbox{$p=10$}, have some scenarios where the objective values are poor for both objectives. These solutions, however, have better bounds on the worst case values of the objectives. These results are due to the fact that the optimal $f(x;\S)$ adapts to the shape of the scalarizing function---a minimal $p$-norm disk that encloses $f(x;\S)$, shown as dotted lines in Figure~\ref{fig-fvals}.

\begin{figure}[hbtp]
\centering
\includegraphics[width=6.5cm]{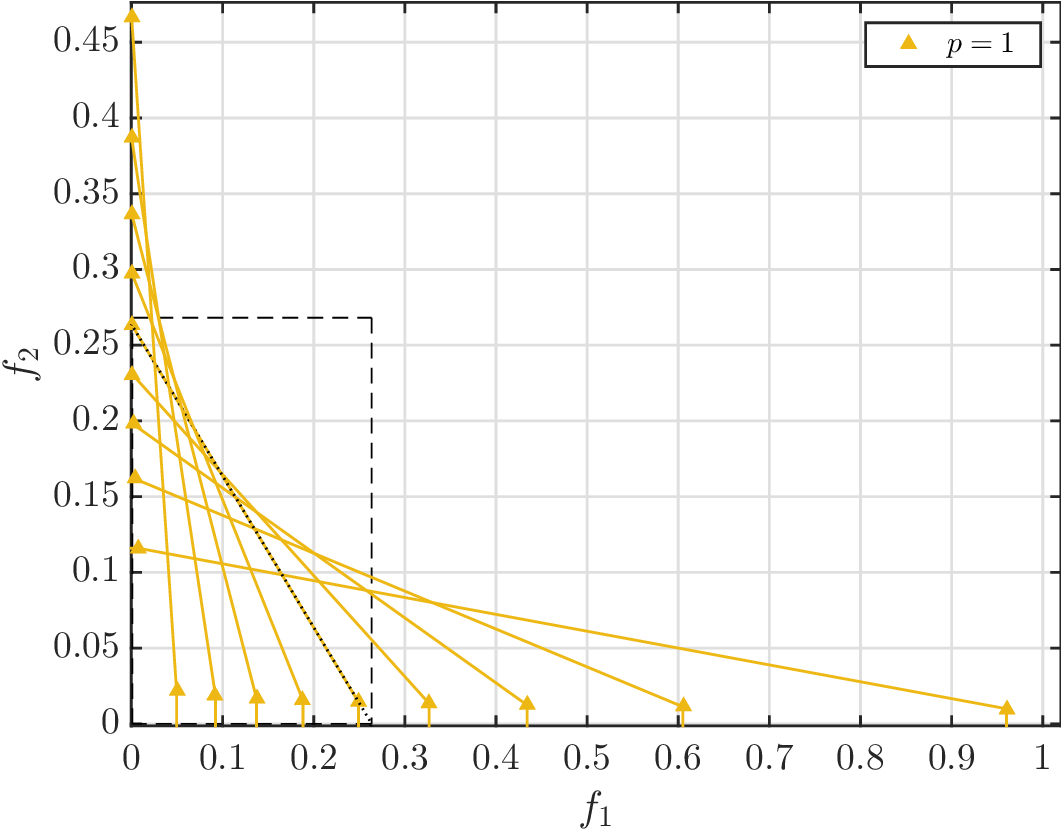}
\includegraphics[width=6.5cm]{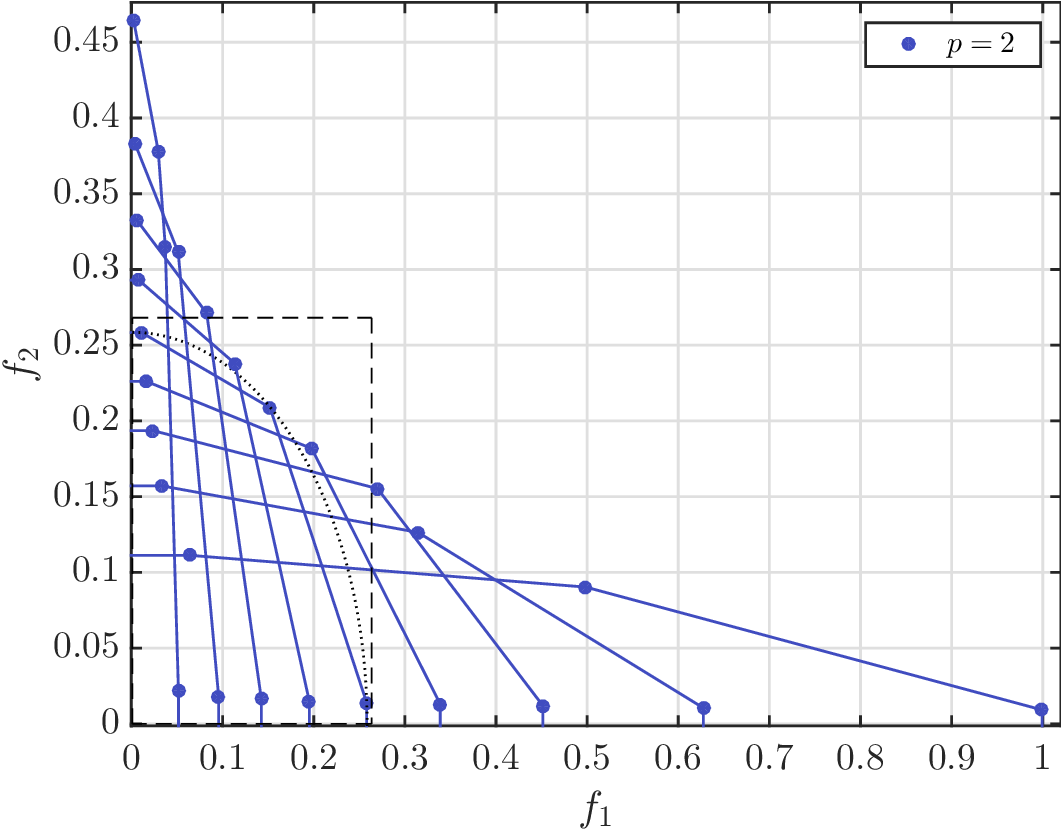} \vspace{5mm} \\
\includegraphics[width=6.5cm]{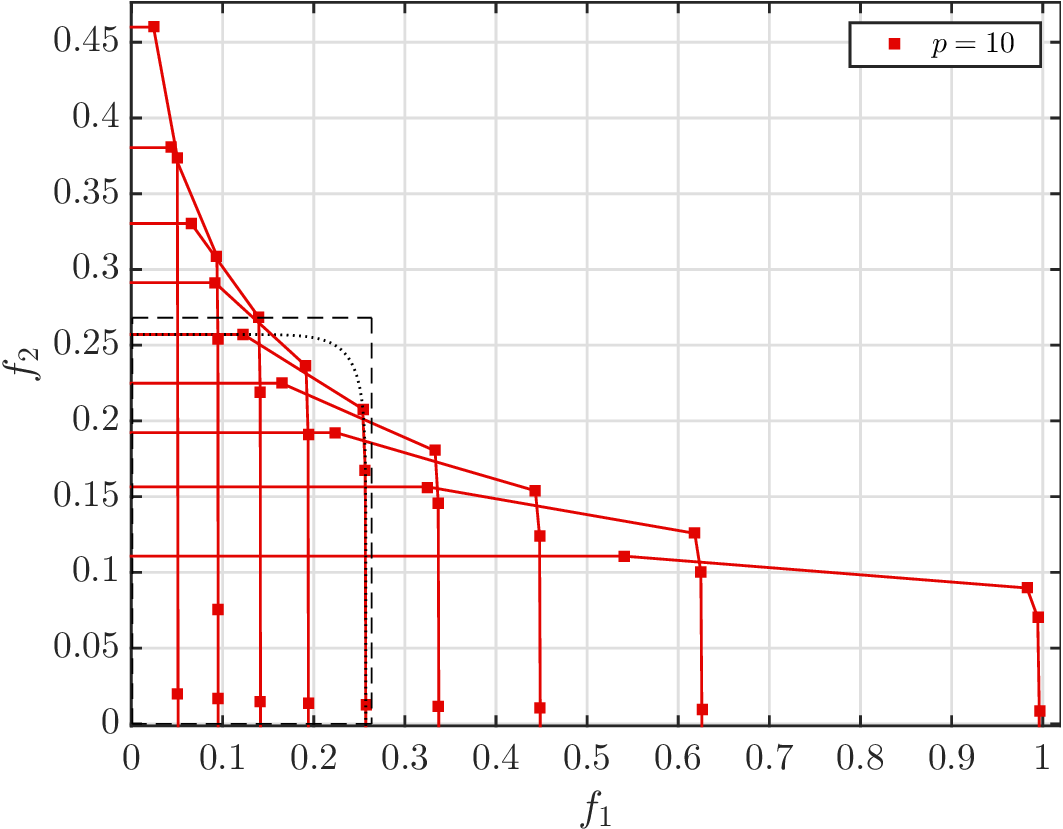}
\includegraphics[width=6.5cm]{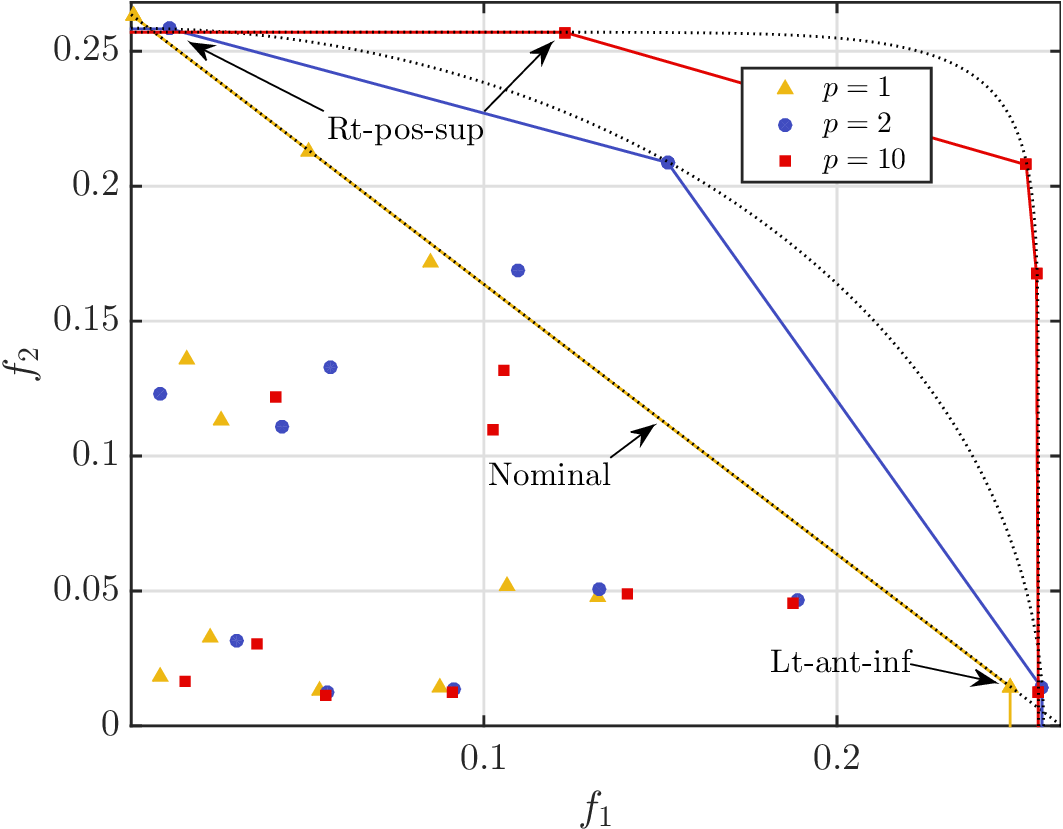}
\caption{Sets \hbox{$\conv(f(x;\S)) - (\Re^n_+ \setminus \{ 0 \})$} for the optimal $x$ to the test problem, for scalarization with \hbox{$p=1,2,10$} and $w$ varied in equidistant steps. The lower right subfigure shows the sets $f(x;\S)$ for the optimal $x$ to equiweighted scalarization, depicted over the region indicated by a dashed box in the other subfigures. Elements that are duplicates within a tolerance of $1 \%$ are here suppressed for legibility (the dose is almost invariant to shifts of the beams along their axis directions). The dotted lines are level curves of the scalarizing functions $u$. The function values are normalized to their maximum value over the solutions of all scalarizations.}
\label{fig-fvals}
\end{figure}

A natural question, given the fact that the choice of scalarizing function influences the shape of the optimal $f(x;\S)$, is to ask which scalarization is preferable. Our results---although not exhaustively characterizing the solutions that can be found by convex scalarizations---indicate that the linear scalarization (\hbox{$p=1$}) is advantageous compared to the scalarizations with positive curvature (\hbox{$p=2$} and \hbox{$p=10$}) because scenarios where both objectives take on poor objective values are avoided at only a minor sacrifice in the objectivewise worst case value of each objectives. 

The differences between \hbox{$p=1$} and \hbox{$p=10$} are further illustrated in Figure~\ref{fig-doses}, which depicts the dose distributions associated with the results in the lower right subfigure of Figure~\ref{fig-fvals}. Ideally, the $95 \%$ isodose region should cover the tumor contour in all scenarios. The left panel shows that the tumor coverage is satisfactory for both solutions in the scenario when no error occurs. Both solutions also underdose the tumor if the left anterioinferior shift occurs (the target contour falls outside the $95 \%$ isodose region in the middle panel). If the right posterosuperior shift occurs (right panel), then the solution for $p=1$ covers the tumor whereas the solution to \hbox{$p=10$} does not. This result is a consequence of that scalarization with $p=10$ gives little incentive to improve tumor coverage beyond the level of tumor coverage in the worst scenario, even in scenarios where improving tumor coverage is not in conflict with sparing of the rectum.

\begin{figure}[hbtp]
\centering
\begin{tabular}{cccc}
 & Nominal scenario & Left anteroinferior shift & Right posterosuperior shift \\
\cmidrule(lr){2-2} \cmidrule(lr){3-3} \cmidrule(lr){4-4}
\begin{sideways}\parbox{2.5cm}{\centering $p=1$}\end{sideways} 
& \includegraphics[width=4.3cm]{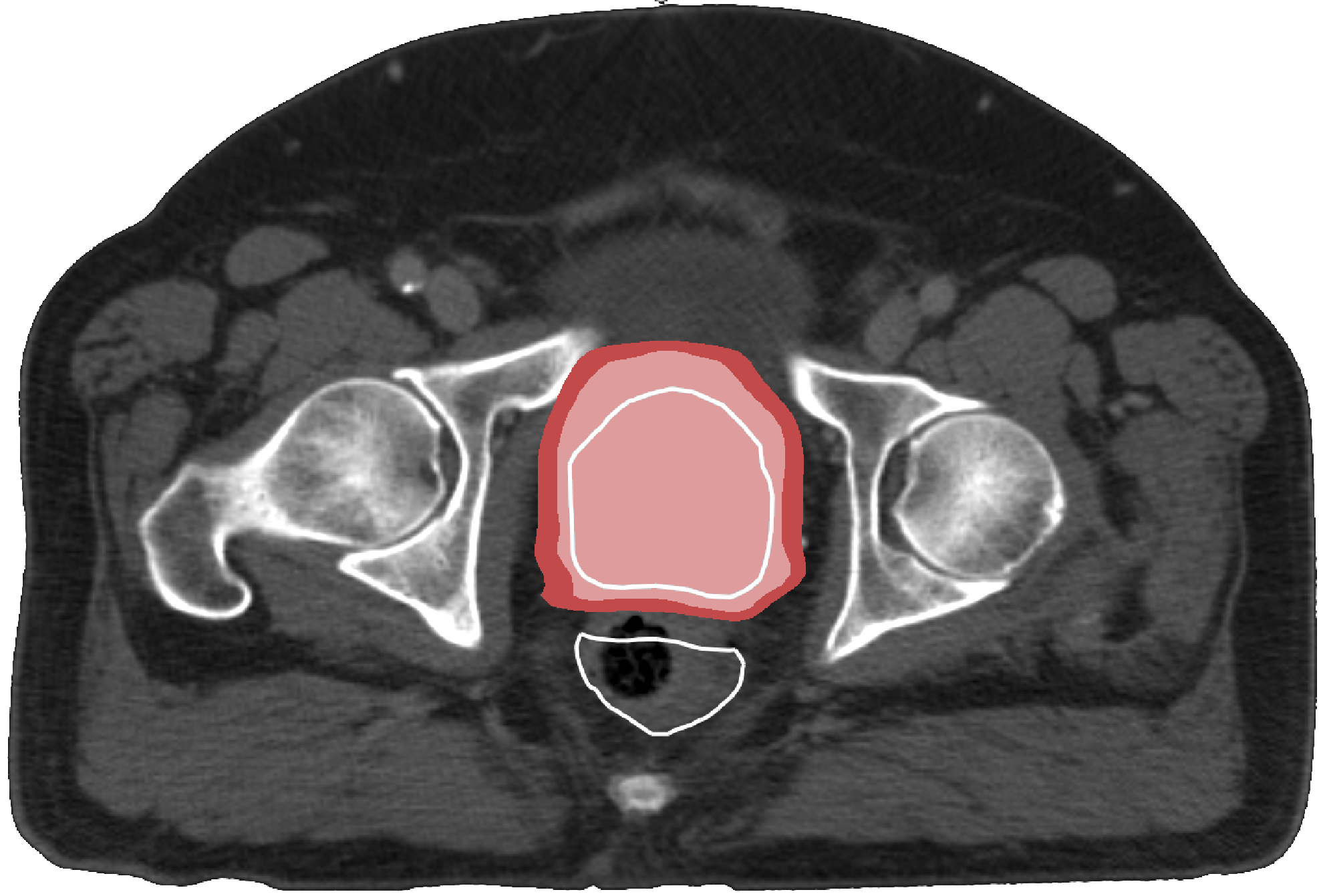} & \includegraphics[width=4.3cm]{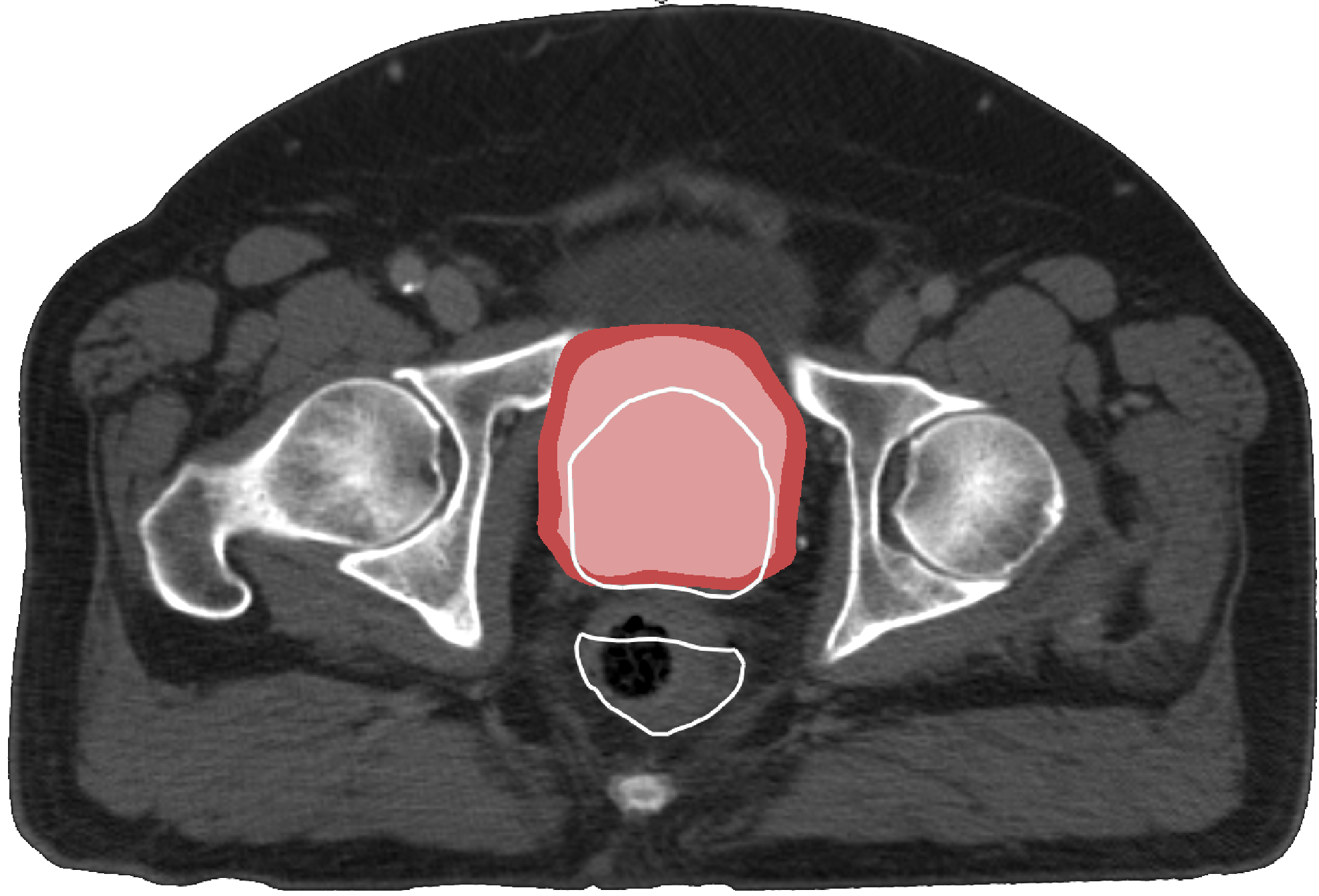} & \includegraphics[width=4.3cm]{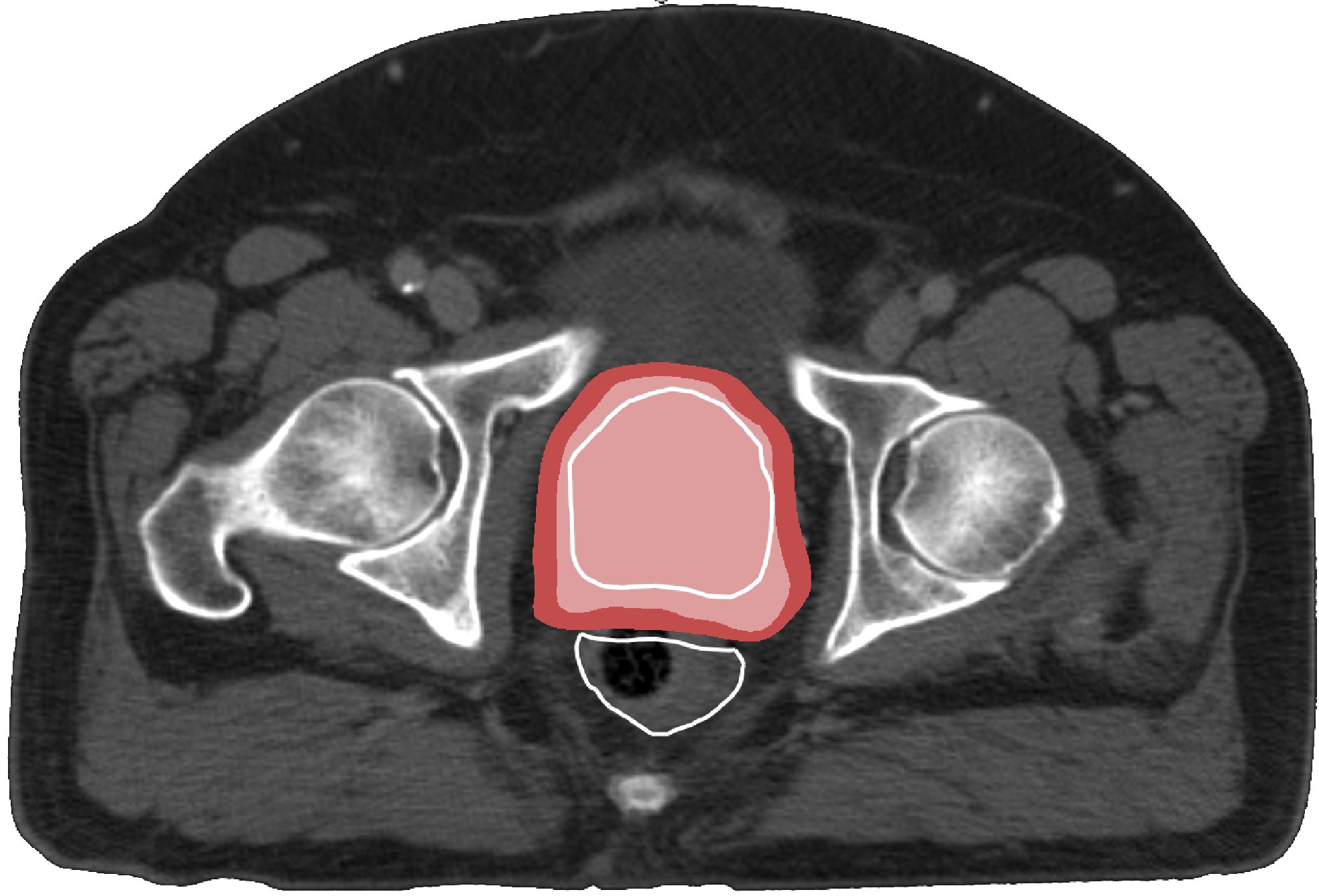} \\
\begin{sideways}\parbox{2.5cm}{\centering $p=10$}\end{sideways} 
& \includegraphics[width=4.3cm]{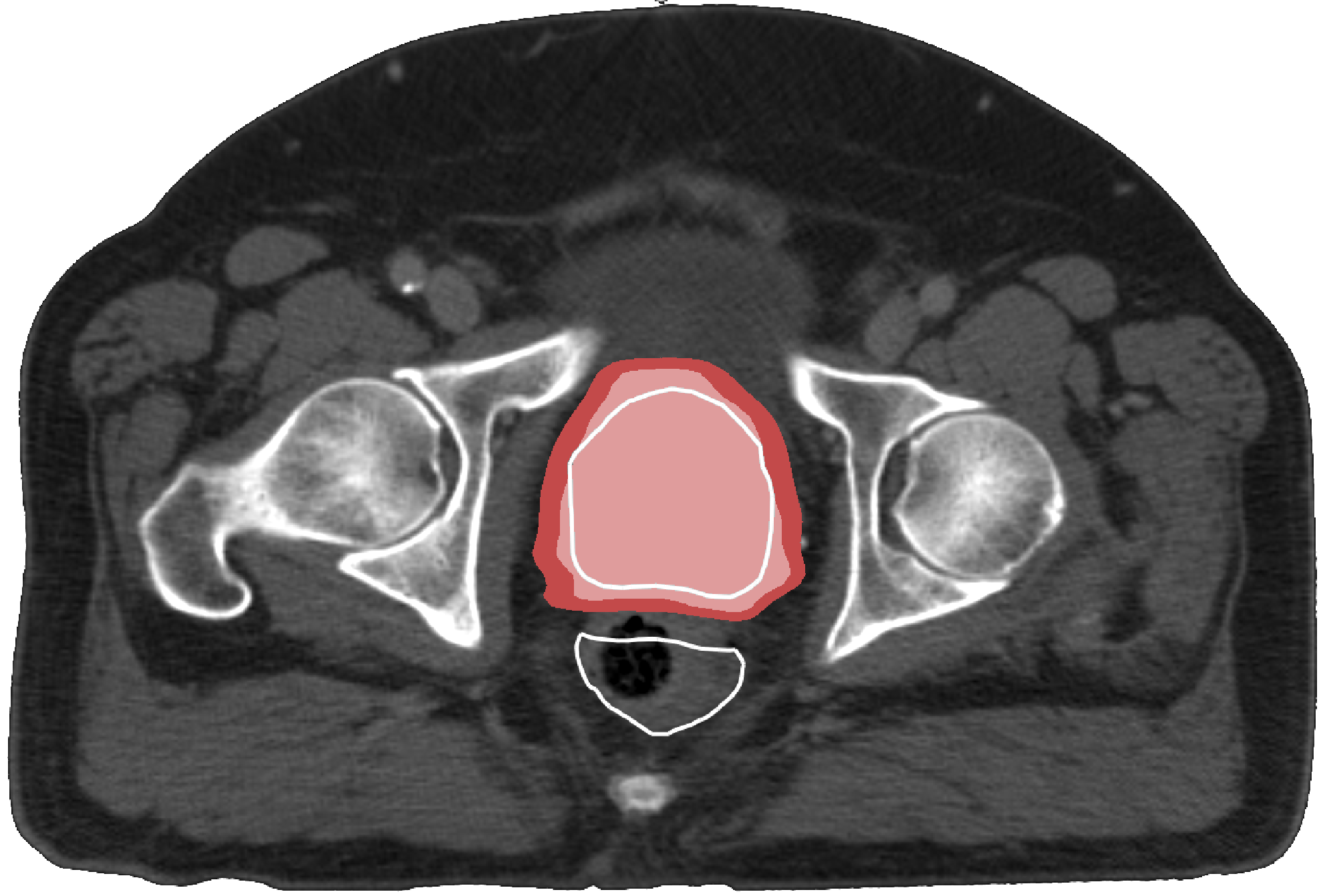} & \includegraphics[width=4.3cm]{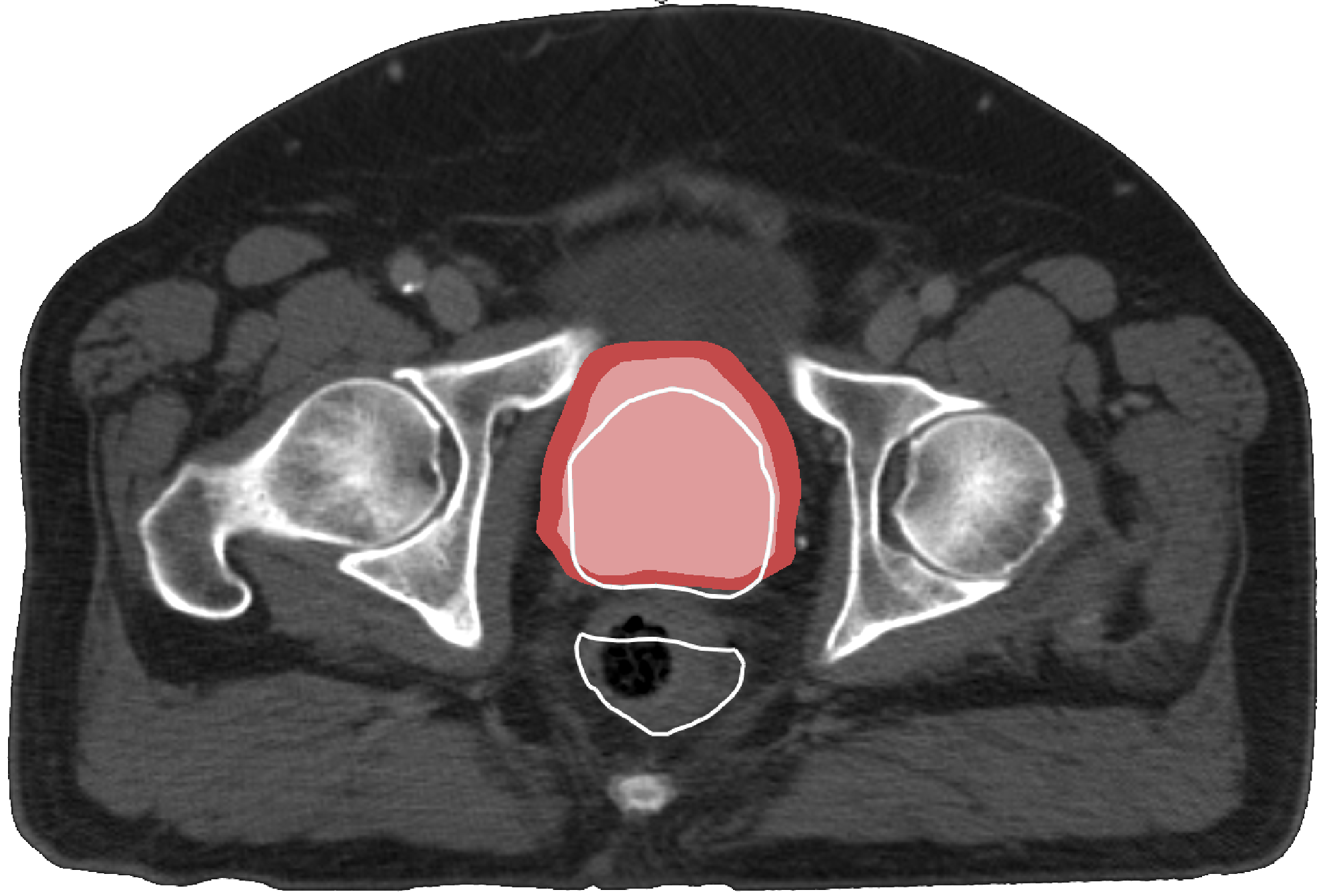} & \includegraphics[width=4.3cm]{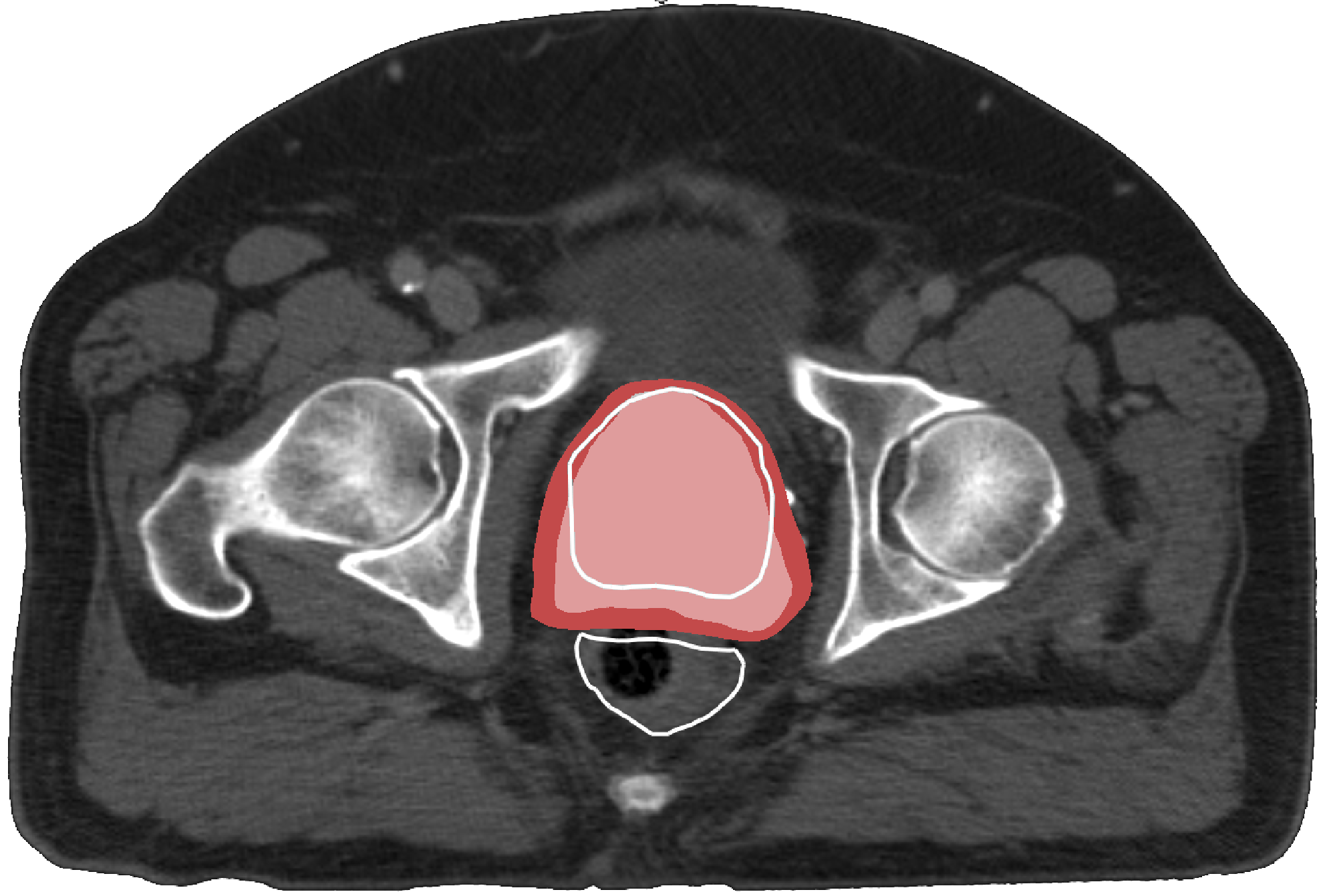} \\
\end{tabular}
\caption{Radiation dose associated with the solutions to \hbox{$p=1$} and \hbox{$p=10$} from the lower right subfigure of Figure~\ref{fig-fvals}, depicted in the nominal scenario and two shift scenarios. The white circular contour indicates the tumor volume, the triangular contour below indicates the rectum, the region in dark red indicates the volume that receives $75 \%$ of the prescribed dose or more, and region in light red indicates the volume that receives $95 \%$ of the prescribed dose or more. The depicted views are transversal cuts of a patient oriented head-first supine, meaning that the anterior direction is up in the picture, the posterior direction is down, the inferior direction is outward from the picture, and the superior direction is inward. The patient's left is right in the picture and the patient's right is left.}
\label{fig-doses}
\end{figure}

\section{Discussion}
Our numerical results show that in robust multiobjective optimization, it is important to carefully select a scalarization that is in line with one's preferences, something that is not necessary in deterministic multiobjective optimization. To clarify, a weighted $p$-norm scalarization (similar to~\eqref{u}) can in the deterministic case find the entire efficient set to convex problems regardless of the choice of $p$~\citep[page 81]{sawaragi85}. If multiple scenarios exist, then the scalarized problem finds a solution $x$ such that the shape of $f(x;\S)$ adapts to the level set of the scalarizing function, meaning that all robust efficient solutions in general cannot be found by a single choice of $p$. For our example, we argued that scalarization with the 1-norm is preferable to a $p$-norm that is closer to the maximum norm, because the 1-norm allows different objectives to take on poor values in different scenarios, and thereby finds solutions that can exploit correlation between the objectives. To exemplify, the solution to the 1-norm scalarized problem of our example exploited that a high dose on the anterior side of the tumor cannot become displaced into the rectum. This high dose region contributes to a robust tumor coverage without posing a risk for rectal toxicity. The solution for \hbox{$p=10$} did not exploit this fact, and thereby achieved a less robust tumor coverage.

A change from 1-norm to 10-norm scalarization for the test problem only resulted in a minor benefit in the worst case values of each objective considered independently. A natural question is whether it is possible to predict the magnitude of the benefit of a change from 1-norm to $\max$-norm scalarization. Given an optimal solution $x^*$ to the 1-norm problem, a trivial bound on the possible improvement in objective \hbox{$i = 1,\ldots,n$} is \hbox{$\max_{s \in \S}\{ f_i(x^*;s) \} - \min_{s \in \S} \{ \rho_i(f(x^*;s)) \}$}, where $\rho_i(y)$ is the $i$th component of $y$ projected along the $i$th dimension onto the 1-norm problem's optimal hyperplane bounding $f(x^*;\S)$ (the dotted line in Figure~\ref{fig-fvals}). An interesting question is whether better bounds can be found for problems with special structure.

The importance of the choice of scalarizing function holds when the uncertainty jointly affects the objectives. If the uncertainty instead materialized in an objectivewise manner, meaning that $\S$ is given by \hbox{$\S = \S_{\mathrm{obj}}^n$}, where $\S_{\mathrm{obj}}$ is the objective-specific uncertainty set, then the choice of scalarizing function is not crucial, just as in the deterministic case. This property stems from the fact that $f(x;\S)$ then forms a hyperrectangle in $\Re^n$, see Figure~\ref{fig-objective-wise}. To any increasing scalarizing function, there is no difference between the solution depicted in the figure and a solution consisting solely of the point $f(x;(s_2,s_3))$. The scalarized problem can therefore be posed as a deterministic multiobjective problem with the objectives \hbox{$\max_{s \in \S_{\textrm{obj}}} f_1(\cdot; s),\ldots, \max_{s \in \S_{\textrm{obj}}} f_n(\cdot; s)$}~\citep[Theorem 5.4]{ehrgott14}. In this case, the model of~\citet{kuroiwa12} is equally applicable as the formulation we consider, with the added advantage that it fits within the deterministic multiobjective framework.

\begin{figure}[hbtp]
\centering
\includegraphics[width=6.5cm]{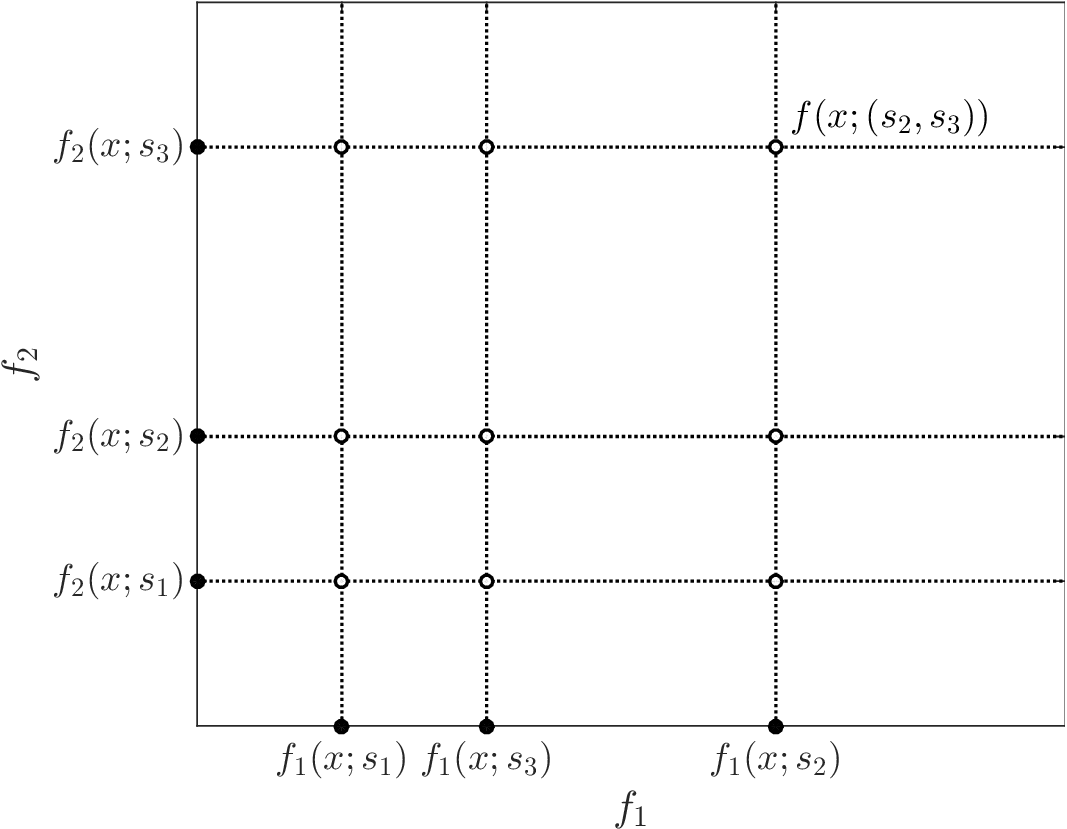}
\caption{Image $f(x;\S)$ (white circles) of a solution $x$ to an uncertain multiobjective problem with two objectives $f_1$ and $f_2$, each of which has uncertainty set \hbox{$\S_{\mathrm{obj}} = \{s_1, s_2, s_3\}$}, that are independently affected by the uncertainty, in the sense that the uncertainty set $\S$ for $f$ is given by \hbox{$\S = \S_{\mathrm{obj}}^2$} and \hbox{$f(x,(s;s')) = (f_1(x;s)~f_2(x;s'))^T$}.}
\label{fig-objective-wise}
\end{figure}

A scalarizing function's influence on the optimized solution can also be understood from the close connection between scalarization and a decision maker's utility function. If the decision maker's preferences follow a utility function to be maximized over $f(\X)$, then the negative of this function would be the ideal scalarizing function. Rational decision makers are generally believed to have strongly decreasing utilities~\citep{rosenthal85}. The scalarizing function should thereby be strongly increasing. By our sufficient conditions for robust efficiency, the robust efficient set thus contains all solutions that a rational decision maker would consider as optimal. If the decision maker is also assumed to have a risk-averse preference, meaning that the increase in utility brought about by a given decrease in objective value is smaller the smaller the objective value is, then the utility function is concave relative to the origin. The scalarizing function should thereby be convex. By our sufficient conditions for convex hull efficiency, the convex hull efficient set thus contains all solutions that a rational and risk-averse decision maker would consider as optimal.

\section{Conclusions}
In this paper, we have arrived at the following theoretical results:
\begin{itemize}
\item \textbf{Robust efficiency:} Each robust efficient solution is optimal to minimization of some strictly increasing scalarizing function (Theorem~\ref{thm-nec-robust-mco}), and any optimal solution to minimization of a strongly increasing scalarizing function is robust efficient (Theorem~\ref{thm-suff-robust-mco}). Not all robust efficient solutions can be found by minimization of a strictly increasing and convex function, even for convex problems (Example~\ref{ex-rob-is-not-optimal-to-convex}). 

\item \textbf{Convex hull efficiency:} The set of convex hull efficient solutions is a subset of the set of robust efficient solutions (Proposition~\ref{che-is-efficient}). In the general case, the subset is proper (Example~\ref{ex-efficient-is-not-che}). Each convex hull efficient solution is optimal to minimization of some strictly increasing and convex scalarizing function (Theorem~\ref{thm-nec-che-robust-mco}), and any optimal solution to minimization of a strongly increasing and convex scalarizing function is convex hull efficient (Theorem~\ref{thm-suff-che-robust-mco}). Not all convex hull efficient solutions can be found by minimization of strictly increasing and linear functions, even for convex problems (Example~\ref{ex-che-is-not-optimal-to-linear}). There is no more general subclass of the robust efficient solutions than the convex hull efficient solutions that allows the necessary conditions to be stated with respect to convex scalarizing functions (Proposition~\ref{prop-che-is-largest-class}).
\end{itemize}
These results lead to an ontology for Pareto efficiency according to Table~\ref{tab-summary}. 

\begin{table}[htbp]
\begin{minipage}{\textwidth}  
\caption{Necessary and sufficient conditions on the scalarizing function for Pareto efficiency.}
\scriptsize
\centering
\begin{tabular}{llll}
\toprule
\begin{tabular}{@{}l@{}} Efficiency concept \\ ~ \end{tabular} & \begin{tabular}{@{}l@{}} Necessary condition \\ ~ \end{tabular} & \begin{tabular}{@{}l@{}} Sufficient condition \\ ~ \end{tabular} & \begin{tabular}{@{}l@{}} Necessary condition \\ for convex problems \end{tabular} \\
\midrule
Deterministic & Strictly increasing, convex & Strongly increasing & Strictly increasing, linear \\
\begin{tabular}{@{}l@{}} Robust \\ ~ \end{tabular} & \begin{tabular}{@{}l@{}} Strictly increasing \\ ~ \end{tabular} &  \begin{tabular}{@{}l@{}} Strongly increasing \\ ~ \end{tabular} & \begin{tabular}{@{}l@{}} Counterexample to strictly \\ increasing, convex \end{tabular}\\
\begin{tabular}{@{}l@{}} Convex hull \\ ~ \end{tabular} & \begin{tabular}{@{}l@{}} Strictly increasing, convex \\ ~ \end{tabular} & \begin{tabular}{@{}l@{}} Strongly increasing, convex\footnote{That convexity of the scalarizing function cannot be dispensed with follows from Example~\ref{ex-efficient-is-not-che}.} \\ ~ \end{tabular} & \begin{tabular}{@{}l@{}} Counterexample to strictly\\ increasing, linear\end{tabular} \\
\bottomrule
\end{tabular}\label{tab-summary}
\end{minipage}
\end{table}



\begin{thebibliography}{26}
\expandafter\ifx\csname natexlab\endcsname\relax\def\natexlab#1{#1}\fi
\expandafter\ifx\csname url\endcsname\relax
  \def\url#1{{\tt #1}}\fi
\expandafter\ifx\csname urlprefix\endcsname\relax\def\urlprefix{URL }\fi
\expandafter\ifx\csname urlstyle\endcsname\relax
  \expandafter\ifx\csname doi\endcsname\relax
  \def\doi#1{doi:\discretionary{}{}{}#1}\fi \else
  \expandafter\ifx\csname doi\endcsname\relax
  \def\doi{doi:\discretionary{}{}{}\begingroup \urlstyle{rm}\Url}\fi \fi

\bibitem[{Baumann et~al.(2016)Baumann, Krause, Overgaard, Debus, Bentzen,
  Daartz, Richter, Zips, and Bortfeld}]{baumann16}
Baumann, M., M.~Krause, J.~Overgaard, J.~Debus, S.M. Bentzen, J.~Daartz,
  C.~Richter, D.~Zips, T.~Bortfeld. 2016.
\newblock Radiation oncology in the era of precision medicine.
\newblock {\it Nat.~Rev.~Cancer\/} {\bf 16}(4) 234--249.

\bibitem[{Ben-Tal et~al.(2009)Ben-Tal, El~Ghaoui, and Nemirovski}]{ben-tal09}
Ben-Tal, A., L.~El~Ghaoui, A.~Nemirovski. 2009.
\newblock {\it Robust Optimization\/}.
\newblock Princeton Series in Applied Mathematics, Princeton University Press,
  Princeton, New Jersey.

\bibitem[{Ben-Tal and Nemirovski(1998)}]{ben-tal98}
Ben-Tal, A., A.~Nemirovski. 1998.
\newblock Robust convex optimization.
\newblock {\it Math.~Oper.~Res.\/} {\bf 23}(4) 769--805.

\bibitem[{Bertsimas and Sim(2004)}]{bertsimas04}
Bertsimas, D., M.~Sim. 2004.
\newblock The price of robustness.
\newblock {\it Oper.~Res.\/} {\bf 52}(1) 35--53.

\bibitem[{Bokrantz and Fredriksson(2013)}]{bokrantz13b}
Bokrantz, R., A.~Fredriksson. 2013.
\newblock On solutions to robust multiobjective optimization problems that are
  optimal under convex scalarization.
\newblock {\it preprint arXiv:1308.4616\/}.

\bibitem[{Bortfeld(1999)}]{bortfeld99}
Bortfeld, T. 1999.
\newblock Optimized planning using physical objectives and constraints.
\newblock {\it Semin.~Radiat.~Oncol.\/} {\bf 9}(1) 20--34.

\bibitem[{Chen et~al.(2012)Chen, Unkelbach, Trofimov, Madden, Kooy, Bortfeld,
  and Craft}]{chen12}
Chen, W., J.~Unkelbach, A.~Trofimov, T.~Madden, H.~Kooy, T.~Bortfeld, D.~Craft.
  2012.
\newblock Including robustness in multi-criteria optimization for
  intensity-modulated proton therapy.
\newblock {\it Phys.~Med.~Biol.\/} {\bf 57}(3) 591--608.

\bibitem[{Chuong(2016)}]{chuong16}
Chuong, T. 2016.
\newblock Optimality and duality for robust multiobjective optimization
  problems.
\newblock {\it Nonlinear~Anal.~Theory~Methods~Appl.\/} {\bf 134} 127--143.

\bibitem[{Deb and Gupta(2006)}]{deb06}
Deb, K., H.~Gupta. 2006.
\newblock Introducing robustness in multi-objective optimization.
\newblock {\it Evol.~Comput.\/} {\bf 14}(4) 463--494.

\bibitem[{Ehrgott(2005)}]{ehrgott05}
Ehrgott, M. 2005.
\newblock {\it Multicriteria Optimization\/}.
\newblock Lectures notes in economics and mathematical systems, Springer,
  Berlin-Heidelberg, Germany.

\bibitem[{Ehrgott et~al.(2014)Ehrgott, Ide, and Sch{\"o}bel}]{ehrgott14}
Ehrgott, M., J.~Ide, A.~Sch{\"o}bel. 2014.
\newblock Minmax robustness for multi-objective optimization problems.
\newblock {\it Eur.~J.~Oper.~Res.\/} {\bf 239}(1) 17--31.

\bibitem[{Engelsman et~al.(2013)Engelsman, Schwarz, and Dong}]{engelsman13}
Engelsman, M., M.~Schwarz, L.~Dong. 2013.
\newblock Physics controversies in proton therapy.
\newblock {\it Semin.~Radiat.~Oncol.\/} {\bf 23}(2) 88--96.

\bibitem[{Fiacco and Kyparisis(1986)}]{fiacco86}
Fiacco, A., J.~Kyparisis. 1986.
\newblock Convexity and concavity properties of the optimal value function in
  parametric nonlinear programming.
\newblock {\it J.~Optim.~Theory~Appl.\/} {\bf 48}(1) 95--126.

\bibitem[{Fliege and Werner(2014)}]{fliege14}
Fliege, J., R.~Werner. 2014.
\newblock Robust multiobjective optimization \& applications in portfolio
  optimization.
\newblock {\it Eur.~J.~Oper.~Res.\/} {\bf 234}(2) 422--433.

\bibitem[{Goberna et~al.(2015)Goberna, Jeyakumar, Li, and
  Vicente-P{\'e}rez}]{goberna15}
Goberna, M., V.~Jeyakumar, G.~Li, J.~Vicente-P{\'e}rez. 2015.
\newblock Robust solutions to multi-objective linear programs with uncertain
  data.
\newblock {\it Eur.~J.~Oper.~Res.\/} {\bf 242}(3) 730--743.

\bibitem[{Gunawan and Azarm(2005)}]{gunawan05}
Gunawan, S., S.~Azarm. 2005.
\newblock Multi-objective robust optimization using a sensitivity region
  concept.
\newblock {\it Struct.~Multidiscip.~O.\/} {\bf 29}(1) 50--60.

\bibitem[{Ide and K{\"o}bis(2014)}]{ide14}
Ide, J., E.~K{\"o}bis. 2014.
\newblock Concepts of efficiency for uncertain multi-objective optimization
  problems based on set order relations.
\newblock {\it Math.~Meth.~Oper.~Res.\/} {\bf 80}(1) 99--127.

\bibitem[{Ide et~al.(2014)Ide, K{\"o}bis, Kuroiwa, Sch{\"o}bel, and
  Tammer}]{ide14b}
Ide, J., E.~K{\"o}bis, D.~Kuroiwa, A.~Sch{\"o}bel, C.~Tammer. 2014.
\newblock The relationship between multi-objective robustness concepts and
  set-valued optimization.
\newblock {\it Fixed Point Theory A.\/} {\bf 2014}(1) 83.

\bibitem[{Ide and Sch{\"o}bel(2016)}]{ide16}
Ide, J., A.~Sch{\"o}bel. 2016.
\newblock Robustness for uncertain multi-objective optimization: a survey and
  analysis of different concepts.
\newblock {\it OR Spectrum\/} {\bf 38}(1) 235--271.

\bibitem[{Kuroiwa and Lee(2012)}]{kuroiwa12}
Kuroiwa, D., G.~Lee. 2012.
\newblock On robust multiobjective optimization.
\newblock {\it Vietnam J.~Math.\/} {\bf 40}(2\&3) 305--317.

\bibitem[{Miettinen(1999)}]{miettinen99}
Miettinen, K. 1999.
\newblock {\it Nonlinear Multiobjective Optimization\/}.
\newblock Kluwer Academic, Boston, Massachusetts.

\bibitem[{Rosenthal(1985)}]{rosenthal85}
Rosenthal, R. 1985.
\newblock Principles of multiobjective optimization.
\newblock {\it Decis.~Sci.\/} {\bf 16}(2) 133--152.

\bibitem[{Ruzika and Wiecek(2005)}]{ruzika05}
Ruzika, S., M.~Wiecek. 2005.
\newblock Approximation methods in multiobjective programming.
\newblock {\it J.~Optimiz.~Theory~App.\/} {\bf 126}(3) 473--501.

\bibitem[{Sawaragi et~al.(1985)Sawaragi, Nakayama, and Tanino}]{sawaragi85}
Sawaragi, Y., H.~Nakayama, T.~Tanino. 1985.
\newblock {\it Theory of Multiobjective Optimization\/}.
\newblock Academic Press, Inc., Orlando, Florida.

\bibitem[{Shapiro and Kleywegt(2002)}]{shapiro02}
Shapiro, A., A.~Kleywegt. 2002.
\newblock Minimax analysis of stochastic problems.
\newblock {\it Optim.~Meth.~Softw.\/} {\bf 17}(3) 523--542.

\bibitem[{Wang et~al.(2015)Wang, Liu, and Chai}]{wang15}
Wang, F., S.~Liu, Y.~Chai. 2015.
\newblock Robust counterparts and robust efficient solutions in vector
  optimization under uncertainty.
\newblock {\it Oper.~Res.~Lett.\/} {\bf 43}(3) 293--298.

\end{thebibliography}

\end{document}